\documentclass{ws-ccm}
\usepackage{hyperref}
\usepackage{mathtools}
\usepackage{graphicx}
\usepackage{pdfsync}
\usepackage[usenames]{color}

 \definecolor{labelkey}{gray}{.45}

\newcommand{\carrenoir}{\nopagebreak\hfill$\blacksquare$\par\medskip}

\newcommand{\eps}{\varepsilon}

\newcommand{\R}{\mathbb R}
\newcommand{\C}{\mathbb{C}}
\newcommand{\NN}{\mathbb{N}}

\newcommand{\beq}{\begin{equation}}
\newcommand{\eeq}{\end{equation}}

\newcommand{\ds}{\displaystyle}

\newcommand{\ts}{\textstyle}

\newcommand{\Meas}{\mathfrak{M}}
\newcommand{\grad}{\nabla}

\newcommand{\restr}[1]{\lfloor_{#1}}

\newcommand{\wto}{\rightharpoonup}
\def\geq{\geqslant}
\def\leq{\leqslant}
\def\div{{\rm div}\,}
\def\curl{{\rm curl}\,}
\def\supp{{\rm supp}\,}

\newcommand{\nnn}{\nonumber}

\newcommand{\HH}{\breve{H}_{\div}^1}
\newcommand{\HHr}{\HH(\R^3;\R^3)}
 
\newcommand{\II}{\hat I}

\newcommand{\Aex}{\widehat A}
\newcommand{\Aexp}{\Aex_\perp}

\begin{document}

\markboth{S. Alama \& L. Bronsard \& B. Galv\~ao-Sousa} {Mean field model for thin films}

\title{Singular Limits for Thin Film Superconductors in Strong Magnetic Fields}

\author{STAN ALAMA \& LIA BRONSARD}

\address{Department of Mathematics and Statistics\\ McMaster University\\
Hamilton, ON, Canada
\\
\emph{\texttt{\{alama,bronsard\}@mcmaster.ca}}}

\author{BERNARDO GALV\~AO-SOUSA}
\address{Department of Mathematics\\ University of Toronto\\ Toronto, ON, Canada\\
\emph{\texttt{beni@math.toronto.edu}}}

\maketitle

\begin{abstract}
{\bf Abstract.} 
We consider singular limits of the three-dimensional Ginzburg-Landau functional for a superconductor with thin-film geometry, in a constant external magnetic field. The superconducting domain has characteristic thickness on the scale $\eps>0$, and we consider the simultaneous limit as the thickness $\eps\to 0$ and the Ginzburg-Landau parameter $\kappa\to\infty$.  We assume that the applied field is strong (on the order of $\eps^{-1}$ in magnitude) in its components tangential to the film domain, and of order $\log\kappa$ in its dependence on $\kappa$.  We prove that the Ginzburg-Landau energy $\Gamma$-converges to an energy associated with a two-obstacle problem, posed on the planar domain which supports the thin film.  The same limit is obtained regardless of the relationship between $\eps$ and $\kappa$ in the limit.  Two illustrative examples are presented, each of which demonstrating how the curvature of the film can induce the presence of both (positively oriented) vortices and (negatively oriented) antivortices coexisting in a global minimizer of the energy.

\end{abstract}

\keywords{Partial Differential Equations; Calculus of Variations; Ginzburg-Landau; superconductivity.}

\ccode{Mathematics Subject Classification 2000: 35J50, 35Q56, 49J45.}

\section{Introduction}

In this paper we continue the study of thin-film superconductors begun in our previous paper \cite{AlamaBronsardGS}.  The superconducting sample occupies a domain  $\Omega_\eps\subset\R^3$, 
$$  \Omega_\eps =\{ (x',x_3)\in\R^3: \ 
 x'\in\omega, \ \eps f(x')<x_3<\eps g(x')\},  $$
where $\omega\subset\R^2$ is a bounded regular domain in the plane,
$f,g: \ \omega\to\R$ are smooth functions on $\omega$ with $f(x')<g(x')$ for all $x'\in\omega$, and $\eps>0$.
 We denote by 
$$   a(x') = g(x') - f(x'),  $$
the thickness of the film for given $x'\in\omega$.  

We study minimizers and Gamma-limits of the full three-dimensional Ginzburg--Landau model, for the superconductor $\Omega_\eps$ subjected to a spatially constant external magnetic field, $\mathbf{h}^{ex}\in\R^3$. 
The state of the superconductor is determined via a complex order parameter 
$\mathbf{u}: \ \Omega_\eps\to\mathbb{C}$ and the magnetic vector potential $\mathbf{A}: \ \R^3\to\R^3$, which determines the magnetic field  $\mathbf{h=\nabla\times A}$.  The energy of the configuration 
$(\mathbf{u},\mathbf{A})$ is given by:
\begin{equation*}
{\bf I}_{\eps,\kappa}({\bf u},{\bf A}) := \frac12 \int_{\Omega_{\eps}} \left( |\grad_A {\bf u} |^2 +  \frac{\kappa^2}{2} \bigl( 1 - |{\bf u}|^2 \bigr)^2 \right) \, d{\bf x} + \frac12 \int_{\R^3} |{\bf h} - {\bf h}^{\rm ex}|^2  \, d{\bf x},
\end{equation*}
with Ginzburg--Landau parameter $\kappa>0$ and thickness parameter $\eps>0$, as above.

With the appearance of three parameters in this problem, $\eps,\kappa$, and $\mathbf{h}^{ex}$, it is important to identify limiting regimes which are mathematically interesting and physically relevant.  As in \cite{AlamaBronsardGS} we rescale the domain by $\eps$ in the $x_3$ direction in order to recognize the correct scaling for $\mathbf{h}^{ex}$ in terms of the thickness parameter.
We introduce rescaled quantities as follows:
\begin{align*}
& x = (x',x_3)= (x_1,x_2,x_3) = \bigg({\bf x}_1, {\bf x}_2, \frac{{\bf x}_3}{\eps}\bigg)\in \Omega_1 \\
& A(x) = ({\bf A}_1, {\bf A}_2, \eps {\bf A}_3)({\bf x}), \\
& u(x) = {\bf u}({\bf x}).
\end{align*}
As a result, the order parameter $u$ is defined in a fixed ($\eps$-independent) domain 
$$ \Omega:= \Omega_1=\{ (x',x): \ f(x')<x_3<g(x'), \ x'\in\omega\}.  $$
We denote by $h=\nabla\times A$, and remark that the physical magnetic field
$\mathbf{h(x)}=\left(\eps^{-1}h'(x), h_3(x)\right)$ in the new coordinates.
The energy transforms as follows:
$$\mathbf{I}_{\eps,\kappa}(\mathbf{u},\mathbf{A})
= \eps \widetilde I_{\eps,\kappa}(u,A),  
$$
where 
\begin{align*}
  \widetilde I_{\eps,\kappa}(u,A)  
&=:  \int_{\Omega} \left( \frac12 | (\grad'-iA)u|^2 + \frac{1}{2\eps^2} \left| (\partial_3 -iA_3)u \right|^2 + \frac{ \kappa^2}{4} \bigl( 1 - |u|^2 \bigr)^2 \right) \, dx  \\
&\qquad\qquad	+ \frac12 \int_{\R^3} \left( \left|h_3 - h^{ex}_3\right|^2 + \frac{1}{\eps^2}\left|h' -  {h'}^{ex}\right|^2 \right)  \, dx,
\end{align*}
and where $\grad' = (\partial_1, \partial_2)$, and the rescaled effective external field takes the form
$$   {h}^{ex}= (h_1^{ex}, h_2^{ex}, h_3^{ex}) = 
           \left( \eps \mathbf{h}_1^{ex}, \eps \mathbf{h}_2^{ex},
         \mathbf{h}_3^{ex}\right).
$$

In our previous work \cite{AlamaBronsardGS}, we considered the case where the Ginzburg--Landau parameter $\kappa$ was fixed, in the limit $\eps\to 0$.
By the above transformation, we noted that taking the {\em parallel} component of the applied field $\mathbf{h}'{}^{ex}= O(\eps^{-1})$ gave a critical scaling for the applied field strength.  We determined the $\Gamma$-limit of the energy in the case of critical, subcritical, and supercritical fields.  The most interesting case was (unsurprisingly) the critical case.  With critical parallel magnetic field, we observed an interaction between the parallel field component and the geometry of the domain, and the $\Gamma$-limiting functional was a two-dimensional Ginzburg--Landau-type energy,
\begin{equation}\label{tfilm}
G_\kappa(v;F)= \int_\omega   a(x')\left\{
  \frac12 |(\nabla - i B')v|^2 + {\kappa^2\over 4} (|v|^2-1)^2 
      + {(a(x'))^2\over 24}|{h^{\text{ex}}}'|^2\,|v|^2
\right\} dx', 
\end{equation}
with $B':\R^2\to\R^2$ satisfying $\nabla'\times B'=F$, an effective magnetic field $\vec F= F\,\vec e_3$ acting perpendicularly to the limiting plane of the film,
$$
  F = h_3^{ex}- (h_1^{ex}, h_2^{ex})\cdot\nabla'\left( {f(x')+g(x')\over 2}\right)  .
$$
In this $\Gamma$-limit, the magnetic field is predetermined by the strength and direction of the original applied field $h^{ex}\in\R^3$ and by the geometry of the domain $\Omega_\eps$, and is part of the variational problem for the thin film limit.  In this way, the functional obtained is of the same type as that studied by Chapman, Du, \& Gunzburger \cite{CDG} for thin film superconductors with constant $\kappa$ and bounded $\mathbf{h}^{\text{ex}}$.  One of the attractions of the critical scaling case is that the effective magnetic field is {\it non-constant} when the film's vertical center $\frac12(f(x')+g(x'))$ is curved, leading to some more interesting configurations of vortices (and antivortices) appearing in minimizers of $G_\kappa$ near the lower critical field.

In this paper, we consider the simultaneous limit of the energy as both $\eps\to 0$ and $\kappa\to\infty$.  We choose an exterior applied field which is critical with respect to the thickness parameter $\eps$, and on the scale of the first critical field in $\kappa$,
$$  \mathbf{h}^{ex} = \left( H'{\log\kappa\over \eps}, H_3\log\kappa\right),  $$
in other words, in the rescaled functional we take
\begin{equation}\label{Hlog}  h^{ex}= H\, \log\kappa,  
\end{equation}
where $H=(H', H_3)=(H_1,H_2,H_3)\in\R^3$ is a fixed constant vector (independent of $\eps,\kappa$.)
The choice of an applied field of order $\log\kappa$ is natural for the Ginzburg--Landau model, both in two dimensions (see \cite{SS} and the references contained therein) and in three dimensions (see \cite{ABMont1}, \cite{BJOS}) as it is the critical scale of the magnetic field strength at which vortices become energetically favorable in the sample.  Our results confirm this scaling in the thin film setting as well.
For applied fields of the form \eqref{Hlog}, it is expected that the energy of minimizers of $\tilde I_{\eps,\kappa}$ will be on the order of $[\log\kappa]^2$.  We are thus led to introduce the following normalization, and study the family of functionals
$$   I_{\eps,\kappa}(u,A):= {1\over(\log\kappa)^2}\widetilde I_{\eps,\kappa}(u,A)  $$
and configurations $(u,A)$ with bounded values of $I_{\eps,\kappa}$.

To present our results, we first introduce appropriate function spaces for the configurations $(u,A)$.  For $u$ this is very simple, $u\in H^1(\Omega_1,\mathbb{C})$.  For $A$, we note that far from $\Omega$, we expect the field $h=\curl A$ to relax to the rescaled field 
$h^{\text{\it ex}}$.  We first choose a fixed $\Aex$ with $\nabla\times \Aex=H=(H_1,H_2,H_3)$.  A convenient choice is:
\begin{equation}\label{eq:gauge_choice}
\Aex := \left( H_2x_3 - \frac12 H_3 x_2\, , \, \frac12 H_3 x_1-H_1x_3\,  ,  \, 0\right).
\end{equation}
Note that this choice fixes a gauge for $\Aex$.
Then, we take our $A$ in the following affine space,
\begin{equation}\label{Adef}
 A\in \mathcal{A}:=\left\{ A\in H^1_{loc}(\R^3;\R^3): \ A-\Aex\log\kappa \in \HHr\right\},
\end{equation}
where $\HHr$ is the closure of the space of smooth, compactly supported divergence-free vector fields $F\in C_0^\infty(\R^3;\R^3)$ in the Dirichlet norm,  (see \cite{GP},)
$$\|F\|_{\HHr}=\left[\int_{\R^3} |DF|^2\, dx\right]^2.  $$

With the onset of vorticity, the limiting behavior of the functional $I_{\eps,\kappa}$ must be described in terms of the limiting currents and vorticity measure (Jacobian) rather than the order parameter $u$, as the number of vortices will become unbounded in the limit.  The current (or momentum density) is:
$$   j=j(u) = (i u, \nabla u),  \quad \text{where $(a,b)=\frac12[\bar a b + a\bar b]$,}
$$
and the vorticity, or weak Jacobian, $J=\frac12\nabla\times j$.  For $u\in H^1(\Omega;\mathbb{C})$, $j\in L^2(\Omega;\R^3)$, and so $J$ is defined in the sense of distributions, although in our context it will in fact be measure-valued (see Jerrard \& Soner \cite{JS}.)  It will often be convenient to represent $j$ and $J$ as differential forms, 
\begin{gather*}
  j=j_1\, dx^1 + j_2\, dx^2 + j_3\, dx^3\in \Lambda^1(\R^3), \\
   J= d\, j = J_1 \, dx^2\wedge dx^3
        + J_2 \, dx^3\wedge dx^1 + J_3 \, dx^1\wedge dx^2\in\Lambda^2(\R^3),
\end{gather*}
as the natural mapping between forms and vector fields is an isometry in Euclidean space $\R^3$.
  The $\Gamma$-limit will impose further structure on the limiting momenta and Jacobians, so we define the following domain:
\begin{equation}\label{eq:V}
\mathcal{Z}:= \left\{ j\in L^2(\Omega;\R^3): \ 
   j=(j'(x'),0), \ J:= \frac12\nabla\times j \in \Meas(\Omega;\R^3)\right\},
\end{equation}
where $\Meas(\Omega;\R^3)$ is the space of vector-valued Radon measures on $\Omega$.
Note that for $j\in\mathcal{Z}$, the corresponding Jacobian takes the form $J=(0,0,J_3(x'))$.  We define the functional
\begin{equation}\label{Iinfty}
  I_\infty(j; F) = \begin{cases}
	\ds  \frac12\|a(x')\nabla\times j\|_{\Meas(\omega)} + 
	\frac{1}{2} \int_\omega a(x') \left|j' - B'\right|^2, 	& \text{ if } j \in \mathcal{Z}, \\
	\infty,		& \text{ otherwise,}
	\end{cases}
\end{equation}
where $F=\nabla'\times B'$, and $B':\R^2\to\R^2$.
This is our {\em mean field} model, or {\em vortex density} functional (see \cite{CRS}.)

The main result is that $I_{\eps,\kappa}$ Gamma-convergences to $I_\infty$.  We prove this in the usual two steps:  first, bounded sequences are compact and the limit is lower semicontinuous in the energies:

\begin{theorem}\label{thm:big1}
For \underbar{any} pair of sequences $\eps_n\to 0$ and $\kappa_n\to\infty$, assume $\{(u_{n}, A_n)\}_{n\in\NN} \subset H^{1}(\Omega;\C) \times \mathcal{A}$ satisfy the uniform bound,
 $$
\sup_{n\in\NN}\, I_{\eps_n,\kappa_n}(u_n,A_n) < +\infty,
$$
and define $j_n=j(u_n)=(u_n,d\, u_n)$ and $J_n=\frac12 d\, j_n$.
Then there exists a subsequence (which we continue to denote $\{\eps_n,\kappa_n\}$) 
and $j_*\in\mathcal{Z}$, with $J_*=\frac12\nabla\times j_*$, such that:
\begin{enumerate}
\item along the subsequence,
\begin{gather}
\label{t1}
|u_n|^2\to 1, \qquad\text{in $L^4(\Omega)$,}, \\
\label{t2}
{A_n \over \log\kappa}- \Aex\wto 0, \qquad\text{in $\HHr$,} \\
\label{t3}
{ j_n\over |u_n|\log\kappa} \wto j_*, \qquad\text{in $L^{2}(\Omega;\R^3)$,}\\
\label{t4}
{J_n\over\log\kappa}\wto J_*, \qquad\text{in the weak-$\star$ topology on  $[C^{0,\gamma}(\Omega)]^*$,}
\end{gather}
for all $0<\gamma<1$.
\item Furthermore,
$$  \liminf_{n\to\infty} I_{\eps_n,\kappa_n}(u_n,A_n)
\ge I_\infty (j_*; F_*) 
   + \frac{1}{2} \int_\omega
     { \frac{a^3(x')}{12}} \big| H'\big|^2  \, dx'
,  $$
where
$F_*$ is defined by
\begin{equation}\label{Fdef}
  F_*(x') = H_3 - (H_1, H_2)\cdot\nabla'\left( {f(x')+g(x')\over 2}\right)  ,
  \end{equation}
and $I_\infty(j; F)$ is defined as in \eqref{Iinfty}.
\end{enumerate}
\end{theorem}
It is important to note the (somewhat surprising) fact that the same limit is obtained regardless of how the two parameters $\eps_n\to 0$ and $\kappa_n\to\infty$.  In fact, this is particular to the case where $h^{ex}=O(\log\kappa)$.  For stronger applied fields $h^{ex}\gg\log\kappa$ the character of minimizers will depend on the relationships between $\kappa,\eps$, and $h^{ex}$.  In particular, in Remark~\ref{remark} we note that when $\eps$ is relatively large compared to $\kappa$ the superconductor will not behave as a thin film at all, and may exhibit a longitudinal vortex lattice, aligned along a horizontal direction. 

The proof of Theorem~\ref{thm:big1} is the content of section~2.
The second part of the Gamma convergence result is the construction of recovery sequences:

\begin{theorem}[]\label{Glimsup}
Let $j \in \mathcal{Z}$ and consider any sequences $\eps_n,\kappa_n$ such that $\eps_n \to 0$ and $\kappa_n\to\infty$.
Then there exists a sequence $\{(u_n, A_n)\} \subset H^1(\Omega;\C)\times\mathcal{A}$,  satisfying
\begin{align*}
& \frac{j_n}{\log \kappa_n} \to j\text{ in  $L^p(\Omega)$, for all $p<2$,} \\
&  \frac{J_n}{\log\kappa_n} \to  J:=\frac12 \nabla\times j \ \text{ weakly in \ $\Meas(\Omega;\R^3)$, and strongly in $(C_0^\gamma(\Omega))'$, $0<\gamma<1$},
\end{align*}
with $j_n:= (iu_n, du_n)$ and $J_n:=\frac12 d\,j_n$.
Moreover, 
$$
\liminf_{n \to \infty} I_{\eps_n,\kappa_n}(u_n,A_n) 
	\leq I_\infty(j_*;F_*).
$$
\end{theorem}
We will prove Theorem~\ref{Glimsup} in section~3.  The construction is essentially two-dimensional, but since the effective applied magnetic field $F_*=\nabla\times B'_*$ is non-constant the procedure is somewhat different from the standard approaches with a constant applied field.

An immediate consequence of the Gamma convergence of $I_{\eps,\kappa}$ to $I_\infty$ is the convergence of minimizers.  To better understand the properties of minimizers of the mean-field limit $I_\infty$ we use con¥vex duality to obtain an equivalent formulation of the problem as a variational inequality.

\begin{proposition}\label{dualprob}
Assume  $j\in \mathcal{Z}$ is a global minimizer of $I_\infty$, and $B_*:\R^2\to\R^2$ with $\nabla\times B_*=F_*$. Then 
$j=B_*-{1\over a}\nabla^\perp\zeta$ where $\zeta\in H^1_0(\omega)$ solves the minimization problem:
\begin{equation}\label{obprob}
\min_{\xi\in H^1_0(\omega) \atop |\xi|\le a(x)/2} \mathcal{E}_\infty(\xi), \quad\text{with}\quad
 \mathcal{E}_\infty(\xi):=
   \int_\omega \left[{1\over a(x)}   |\nabla \xi|^2 - F_*\, \xi\right].
\end{equation}
\end{proposition}
The equation for $\zeta$ is a two-obstacle problem for Poisson's equation, and leads to solutions with free boundaries on the coincidence sets, where $|\zeta|={a(x)\over 2}$.  The coincidence sets form the support of the measure $J$, and indicate the regions of nonzero vorticity of minimizers in the simultaneous $\kappa\to\infty$, $\eps\to 0$ limit.  Obstacle problems of a similar type were obtained in the limit of the two-dimensional Ginzburg-Landau functional with constant vertical applied field by Sandier \& Serfaty \cite{SSfbp}, and in a non-homogeneous setting involving pinning of vortices in \cite{ASS}.  In these papers, there is a single obstacle, as the analogue of the solution $\xi$ is constrained on only one side, and in the 2D setting the equation is of Helmholtz (rather than Poisson) type.  

  This more concrete characterization of the limiting problem gives us a better idea of what minimizers look like for applied fields on the order of the first critical field.  In section~4 we revisit two examples (for superconducting films which approach a disk in the $\eps\to 0$ limit,) which we introduced in our first paper.  In the first example, vortices accumulate in two symmetrically placed subdomains in the disk, one containing positively oriented vortices, and the other antivortices (with negative winding.)  In particular, this implies the rather surprising conclusion that vortices and antivortices can coexist in global minimizers of the three-dimensional Ginzburg--Landau model with a constant applied field.

The second example illustrates the phenomenon of concentration on curves and annular subdomains, which occurs even in a simply connected domain in the thin-film limit.  This example is quite appealing in that the free boundary problem obtained is radially symmetric, and an explicit solution may be calculated, and the changing geometry and topology of the regions of vorticity are explicitly shown.  Previous examples of concentration on curves or annuli were found for minimizers in annular domains (see \cite{AAB}, \cite{AB1}, \cite{AB2}, \cite{Kachmar}, \cite{Rougerie}.) Determining the asymptotic distribution of vortices on curves in the domain requires more delicate estimates very near to the lower critical field; this is done in \cite{ABMi}.

Finally, a different type of thin film problem has recently been studied by Contreras \& Sternberg \cite{ContStern} and Contreras \cite{Cont}.  In their setting, the superconductor is thin shell, built from depositing an $\eps$-thick coating on a fixed two-dimensional surface in $\R^3$.  The limiting problem in this case is a Ginzburg--Landau model on an embedded 2-manifold, and they obtain remarkable results connecting the lower critical field and the appearance of vortices to the geometry of the limiting surface.

Several results on Gamma convergence or the convergence of local minimizers have also been proven for three-dimensional models of superconductivity (or Bose-Einstein condensation) without assuming thin-film geometry.  A recent paper by Baldo, Jerrard, Orlandi, \& Soner \cite{BJOS} proves Gamma convergence under very general hypotheses.  Results on stable vortex solutions for these models in various specific geometries may be found in \cite{MontSternZ}, \cite{JMS}, \cite{J_BEC}, \cite{ABMont1}, \cite{ABMont2}.

\section{Compactness and lower bound}

In this section we prove two parts of the Gamma-convergence result, the compactness of energy-bounded sequences and the lower bound inequality.
As always when dealing with the Ginzburg--Landau functionals, the gauge invariance of the functionals is an issue.  The choice of the space $\mathcal{A}$ fixes a gauge (see \eqref{eq:gauge_choice} and \eqref{Adef} above.)  We will require the following essential lemma:


\begin{lemma}[Lemma 3.1 in \cite{GP}]\label{lem:GP}
Let $g \in L^2(\R^3;\R^3)$ such that $\div g = 0$ in $\mathcal{D}'(\R^3)$.
Then there is a unique $B \in H^1(\R^3;\R^3)$ such that $\grad \times B = g$ and $\div B = 0$.

\end{lemma}
 As a consequence, it follows that
 $$  \| B \|_{\HH}=\left[\int_{\R^3} |\nabla\times B|^2 \, dx \right]^{\frac12}  $$
 is equivalent to the usual (Dirichlet) norm on the space $\HHr$.


\begin{proof}[ of Theorem \ref{thm:big1}]

Let $K := \sup_{n \in \NN} I_{\eps_n,\kappa_n}(u_n,A_n) < \infty$.
From the energy bound we immediately deduce that,
\begin{gather} \label{lb1}
   \int_\Omega (|u_n|-1)^4 \le \int_\Omega (|u_n|^2-1)^2
         \le {K(\ln\kappa_n)^2\over \kappa_n^2} \to 0, \\
    \frac{h'_n}{\log \kappa_n} - H'  \to 0 \quad \text{ in } L^2(\R^3;\R^2). \label{lb2}     
\end{gather}
In particular, $|u_n|\to 1$ in $L^4(\Omega)$.  The vertical components of the magnetic field are bounded via the energy bound, and thus along a subsequence (which we continue to denote $u_n,A_n$),
we may conclude the weak convergence,
 \begin{equation}
 \frac{{h_3}_{n}}{\log \kappa_{n}} - H_3 \wto \ell  \quad \text{ in } L^2(\R^3).  \label{lb3}
\end{equation}
As the vectors $[{h_n\over\log\kappa_n}-H]\wto \tilde H:=(0,0,\ell)$ in $L^2(\R^3;\R^3)$, and each $\div h_n=0$ (in the sense of distributions,) we may conclude that the limit is also divergence-free, $\div \tilde H=0$.
As a consequence of \eqref{lb3} and Lemma~\ref{lem:GP}, we also conclude that there exists $\tilde A\in \HHr$ with $\nabla\times \tilde A= \tilde H=(0,0,\ell)$ and
\begin{equation}\label{Astar}  {A_n \over \log\kappa_n}- \Aex \to \tilde A,\end{equation}
weakly in $\HHr$, and in the norm topology on $L^p(\Omega)$, $1\le p<6$.
Since $\tilde H=(0,0, \ell) \in L^2(\R^3;\R^3)$ with $\div \tilde H=0$, we conclude that $\partial_3 \ell =0$ (in the sense of distributions,) and thus $\ell=0$, and also $\tilde H=0=\tilde A$ (by Lemma~\ref{lem:GP}.)  In particular, \eqref{Astar} implies
\begin{equation}\label{lb4}
{A_n \over \log\kappa_n}- \Aex \to 0 \quad
\text{weakly in $\HHr$, and in the norm on $L^p(\Omega)$, $1\le p<6$.}
\end{equation}

To obtain the lower bound we adapt the argument of  \cite{SS_product}.
Expanding the quadratic term in the energy bound, we obtain:
\begin{align*}
2K &\ge 2I_{\eps_n,\kappa_n}(u_n, A_n) \\
	& \geq (\log\kappa_n)^{-2}\int_{\Omega} \left( |\grad' u_n|^2 - 2 A'_n \cdot (iu_n,\grad'u_n) + |u_n|^2 |A'_n|^2 \right)  dx\\
	&\ge 
	\int_{\Omega} \left( 
	    \frac12  \left|\frac{\grad' u_n}{\log \kappa_n} \right|^2 
	      - \left|\frac{A'_n}{\log \kappa_n}\right|^2 |u_n|^2 \right)
	      \\
	 &= \int_{\Omega} \left( 
	    \frac12  \left|\frac{\grad' u_n}{\log \kappa_n} \right|^2 
	      - \left|\frac{A'_n}{\log \kappa_n}\right|^2 (|u_n|^2-1)
	     -  \left|\frac{A'_n}{\log \kappa_n}\right|^2
	       \right)
\end{align*}
By \eqref{lb4}, the last term is bounded, and
$$  \int_\Omega \left|\frac{A'_n}{\log \kappa_n}\right|^2 
   \left||u_n|^2-1\right| \le 
   \left\| {A'_n\over \log\kappa_n}\right\|_{L^4}^2
    \left\| |u_n|^2-1 \right\|_{L^2} \le C{\log\kappa_n\over
    \kappa_n},
$$
by the energy bound and the $L^p$ boundedness of $A'_n\over \log\kappa_n$.
Thus we have
\begin{equation}\label{eq:gradprime}  \int_{\Omega}  \left|\frac{\grad' u_n}{\log \kappa_n} \right|^2
\le C,
\end{equation}
with constant $C$ depending on the energy bound $K$.  
By a similar calculation, we may also obtain the estimate
$$
\int_{\Omega} {1\over\eps_n^2}\left|
{\partial_3 u_n\over\log\kappa_n}\right|^2 \, dx \le C,
$$
and so we have strong convergence in the $x_3$-direction,
\begin{equation}\label{eq:strongconv}
{\partial_3 u_n\over\log\kappa_n} \to 0 \quad \text{ in } L^2(\Omega;\C)
\end{equation}

We now turn to the currents, $j_n:= (iu_n,\nabla u_n)$.  Following \cite{JS} we normalize the currents as follows,
$$   \tilde j_n:= {j_n\over |u_n|\log\kappa_n} 
     = {(iu_n,\nabla u_n)\over |u_n|\log\kappa_n}.  $$
We observe that each component of 
  $\tilde j_n=(\tilde j_{1,n}, \tilde j_{2,n}, \tilde j_{3,n})$ is (for fixed $n$) pointwise (a.e.) bounded,
\begin{equation}\label{eq:jtilde}
   |\tilde j_{k,n}| \le {|\partial_k u_n|\over\log\kappa_n} , \quad k=1,2,3, 
\end{equation}
so $\tilde j_n$ is well defined almost everywhere in $\Omega$.
Moreover, from \eqref{eq:gradprime} and \eqref{eq:strongconv} it follows that there exists $j=(j',0)\in L^2(\Omega;\R^3)$ such that (along a subsequence)
$$  \tilde j'_n \wto j', \qquad \tilde j'_{3,n}\to 0 $$
in $L^2(\Omega)$.  Writing
$$  {j_n\over\log\kappa_n} = \tilde j_n + (|u_n|-1) \tilde j_n,  $$
we recall that $(|u_n|-1)\to 0$ in $L^4(\Omega)$ (see \eqref{lb1},) and thus obtain that 
\begin{equation}\label{eq:current}
   {j'_n\over\log\kappa_n}\wto j', \qquad {j_{3,n}\over\log\kappa_n}\to 0
\qquad\text{in $L^{4/3}(\Omega)$.}
\end{equation}
We will require this fact later on in determining the limit functional.

We continue as in the proof of Theorem 2 of \cite{SS_product}, with some simplification due to the thin film limit, and our modification of the currents.  Let $e_1, e_2, e_3$ be the standard basis in $\R^3$, and define vector fields $X_k=h_k e_k$, $k=1,2,3$, with $h_k\in C_0(\Omega)$ and $|h_k(x)|\le 1$ for all $x\in\Omega$, $k=1,2,3$.  By \eqref{eq:gradprime}, we have
\begin{equation}\label{eq:weakconv}  
{|X_k\cdot\nabla' u_n|\over \log\kappa_n} = 
        {|h_k\,\partial_k u_n|\over \log\kappa_n} \wto \phi_{X_k}, \quad k=1,2,3,
\end{equation}
weakly in $L^2(\Omega)$.  Passing to the weak limit in the bound \eqref{eq:jtilde}, we conclude that
$$  |X_k\cdot j| = |h_k\, j_k| \le \phi_{X_k}, \qquad k=1,2,3, $$
pointwise a.e.~in $\Omega$.  By \eqref{eq:strongconv}, $\phi_{X_3}=0$.
We also define the defect measures $\nu_{X_k}$ corresponding to the weak convergence in \eqref{eq:weakconv}:
\begin{equation}\label{eq:defectmeas}
\left|\frac{X_k\cdot\nabla' u_n}{\log \kappa_n}\right|^2 
  \wto |\phi_{X_k}|^2 + \nu_{X_k} \quad \text{ in the sense of measures},
\end{equation}
for $k=1,2,3$. Because of the strong convergence in \eqref{eq:strongconv}, it follows that $\nu_3\equiv 0$.

For the Jacobians $J_n$, we apply Theorem 1 in \cite{SS_product}:
by the energy bound,
$$ E_{\kappa_n}(u_n;\Omega):= 
   \int_\Omega \left( \frac12 |\nabla u_n|^2
    + {\kappa_n^2\over 4} (|u_n|^2-1)^2\right) \le C[\ln\kappa_n]^2,
    $$
with constant $C$ independent of $n$ (using the estimates \eqref{eq:gradprime}, \eqref{eq:strongconv}, and \eqref{lb1}), we may conclude that
$$
\frac{J_n}{\log \kappa_n} \stackrel{\star}{\wto} J^*,
$$
in the weak${}^*$ topology on $[C_0^{0,\alpha}(\Omega)]^*$, for $0<\alpha<1$.  Moreover, the limiting Jacobian is a Radon measure-valued two-form. 
Furthermore, the same theorem relates the limiting Jacobian to the defect measure $\nu_{X_k}$ via a product formula (see \eqref{prodform} below.)
We prove the following properties of the limiting Jacobian and currents:

\begin{lemma}\label{lem:jacob}
The limiting Jacobian $J^*=\frac12 dj^*$ has the form $J^*= J^*_3\, dx^1\wedge dx^2$ with $J_3=J_3(x')$, and the limiting current $j_*\in \mathcal{Z}$.
\end{lemma}

\begin{proof}[]  We make use of the product formula from \cite{SS_product} in the case where $N_\kappa=O(\ln\kappa)$, which we review here.  Let $E$ be a bounded smooth domain in $\R^3$, and $v_\kappa\in H^1(E,\mathbb{C})$ satisfying
$$  E_\kappa(v_\kappa;E):= \int_E \left( \frac12 |\nabla v_\kappa|^2
    + {\kappa^2\over 4} (|v_\kappa|^2-1)^2\right) \le C[\ln\kappa]^2,
    $$
for constant $C$ independent of $\kappa$.  Let $X,Y$ be continuous, compactly supported vector fields in $E$, and $\nu_X$, $\nu_Y$ the defect measures (defined as in \eqref{eq:defectmeas}) for $v_\kappa$ as $\kappa\to\infty$.  Then, the normalized Jacobians 
${J_\kappa\over\log\kappa}\stackrel{\star}{\wto} J$ in $(C^\gamma_0(\Omega))'$
 for all $\gamma>0$, and the defect measures are related to the limiting Jacobian via:
\begin{equation}\label{prodform}
| \nu_X|(E)\, |\nu_Y|(E) \ge \left| \int_E J(X,Y) \right|^2.
\end{equation}
Here we denote by $|\nu|(E)$ the total variation of the measure $\nu$ over the set $E$.

We note as above that for any $E\subset\Omega$, 
$E_{\kappa_n}(u_n;E)\le [\ln\kappa_n]^2 I_{\eps_n,\kappa_n}(u_n, A_n))\le C[\ln\kappa_n]^2. $
Let $E$ be any open ball contained in $\Omega$, and $X_k=h_k\, e_k$,  with $h_k\in C_0(E)$ and $|h_k|\le 1$, $k=1,2,3$.  Applying the product formula we then obtain,
$$  0= |\nu_{X_1}|^{\frac12}(E) |\nu_{X_3}|^{\frac12}(E)
  \ge \left| \int_E J^*(X_1,X_3) \right| 
     = \left| \int_E h_1\, h_3\, J^*(e_1,e_3) \right| . $$
Taking the supremum over all such $h_1, h_3$, we conclude that, as a Radon measure, $J^*(e_1,e_3)=0$ in the ball $E$.  By an analogous computation with $X_2=h_2\, e_2$ and $X_3$ (as above), we also have
$J^*(e_2,e_3)=0$ in the ball $E$.  This holds for any ball $E\subset\Omega$, and thus these measures vanish identically in $\Omega$, and thus
the Jacobian has the form $J^* =J^*_3\, dx^1\wedge dx^2$. 

Furthermore, since $J_n=\frac12 d j_n$ for each $n$, it follows that $dJ_n=0$ (in the sense of distributions.)  Normalizing by $\log\kappa_n$ and passing to the limit, we retain  $d J^* = 0$, and hence  $\partial_3 J^*_3 = 0$ in $\mathcal{D}'(\Omega)$, so $J^*_3=J_3^*(x')$.
This also implies that (in the sense of $\mathcal{D}'(\Omega)$,)
\begin{align*}
0 = J^*_1 & = \partial_2 j_{*3} - \partial_3 j_{*2} = - \partial_3 j_{*2}, \\
0 = J^*_2 & = \partial_3 j_{*1} - \partial_1 j_{*3} = \partial_3 j_{*1}.
\end{align*}
Thus, the limiting current must have the form
$j_* = \big(j'_*(x'),0\big)$ and $J^*=\frac12\nabla\times j_*\in\Meas(\Omega)$, and hence $j_*\in \mathcal{Z}.$

\end{proof}

It remains to verify the lower bound inequality.  From the definition of the defect measures and the product formula from Theorem~1 of \cite{SS_product},
\begin{align} \nnn
\liminf_{n\to\infty} \int_\Omega  
   \left|{\nabla' u_n\over \log\kappa_n}\right|^2
   & \ge \sum_{k=1,2} \liminf_{n\to\infty} \int_\Omega  
   \left|{X_k\cdot\nabla' u_n\over \log\kappa_n}\right|^2  \\
   \nnn
   &\ge |\nu_{X_1}|(\Omega) + |\nu_{X_2}|(\Omega) 
       + \int_\Omega \left( \phi_{X_1}^2 + \phi_{X_2}^2\right) \\
       \nnn
   &\ge 2 \left| \int_\Omega J^*(X_1,X_2)\right| 
      + \int_\Omega (X_1\cdot j_*)^2 + (X_2\cdot j_*)^2 \\
    &= 2 \left| \int_\Omega h_1\, h_2\, J^*(e_1,e_2)\right| +
          \int_\Omega \left( h_1^2 |j_*\cdot e_1|^2 + h_2^2 |j_*\cdot e_2|^2\right).  \label{eq:bigest}
\end{align}
The above estimate is valid for any $h_k\in C_0(\Omega)$ with $|h_k(x)|\le 1$, $k=1,2$.  We choose these functions to obtain an estimate in terms of the total variation of the measure $J_3=J^*(e_1,e_2)$.  By the Hahn decomposition, we may write $J_3=\mu_+-\mu_-$ for mutually singular, nonnegative finite measures $\mu_+, \mu_-$, supported on the disjoint sets $E_+, E_-\in\Omega$, respectively.  Take sequences $h_{1,i},h_{2,i}\in C_0(\Omega)$ with $|h_{k,i}|\le 1$, $k=1,2$, such that
$$  h_{1,i}\to 1, \qquad  h_{2,i}\to \chi_{E_+}-\chi_{E_-}$$
pointwise a.e.~in $\Omega$.  Passing to the limit $i\to \infty$ on the right hand side of \eqref{eq:bigest} (using the Lebesgue dominated convergence theorem,) we conclude that
\begin{equation}\label{lb10}  \liminf_{n\to\infty} \int_\Omega  
   \left|{\nabla' u_n\over \log\kappa_n}\right|^2 \ge 
      2\,|J^*_3|(\Omega) + \int_\Omega |j_*|^2.
\end{equation}

Finally, we derive the form of the lower bound for the full energy.  First,
from the strong $L^4$ convergence \eqref{lb1} of $|u_n|$, the weak $L^{\frac43}$ convergence \eqref{eq:current} of the normalized currents, the strong $L^p$ ($1\le p<6$) convergence of the vector potentials \eqref{lb4}, and the lower bound \eqref{lb10}, we may conclude that
\begin{align*}
 \liminf_{n\to\infty} 
I_{\eps_n,\kappa_n}(u_n,A_n) 
	&\geq  
	\liminf_{n\to\infty} 
	\frac1{2[\log \kappa_n]^2} \int_\Omega \left(  |\grad' u_n|^2 - 2A_n' \cdot j_n' + |A'_n|^2 
	  + \left(|u_n|^2-1\right) |A_n'|^2 \right) \, dx  \\
&\ge  \|J^*\| + \frac12\int_\Omega\left[ |j_*|^2  -2 \Aex\cdot j_*
+|\Aex|^2 \right]dx
\end{align*}


Since both $J_3=J_3(x')$ and $j_*' = j_*'(x')$, we may integrate out the variable $x_3$, to reduce to a two-dimensional total variation, weighted by the film thickness function $a(x')$, 
$$
\|J_3\|_{\Meas(\Omega)} = \|a(x')J_3\|_{\Meas(\omega)}.
$$
The limiting vector potential $\Aex$ (defined in \eqref{eq:gauge_choice}) is $x_3$-dependent, but this dependence may be averaged out (to produce the desired effective field $F_*$). Indeed,  we decompose ${\Aex}'$ as follows:
$$   {\Aex}' = \left(- \frac12 H_3 x_2\, , \, \frac12 H_3 x_1\right) +
\left( H_2x_3 \, , \, -H_1x_3\right) =: \Aexp + 
\left( H_2x_3 \, , \, -H_1x_3 \right).
$$
Expanding the energy and integrating out $x_3$, we have:
\begin{align}
\int_\Omega |j'_*-{\Aex}'|^2 \, dx
	& = \int_\Omega |j'_*-{\Aexp} - (H_2,-H_1) x_3|^2 \, dx \notag\\
	& = \int_\omega a(x') |j'_*-{\Aexp}|^2 
	- 2 \int_\omega (j'_*-{\Aexp}) \cdot (H_2,-H_1) \int_{f(x')}^{g(x')} x_3 \, dx_3 \, dx' \notag \\
	&\qquad 
	+ \int_\omega \big|H_2,-H_1)\big|^2
	   \int_{f(x')}^{g(x')} x_3^2 \, dx_3 \, dx' \notag\\
	& = \int_\omega a(x') |j'_*-{\Aexp}|^2 
	- 2 \int_\omega a(x'){\ts  \frac{(f+g)}{2}} 
	(j'_*-{\Aexp})
	  \cdot (H_2,-H_1) \, dx'  \notag \\
	&\qquad 
	  + \int_\omega a(x') {\ts \frac{(f^2 + fg + g^2)}{3}}
	  \big|(H_2,-H_1)\big|^2\, dx' \notag\\
	& = \int_\omega a(x') \left|j'_*-{\Aexp} - {\ts \left(\frac{f+g}{2}\right)} (H_2,-H_1) \right|^2  \, dx'
	+ \int_\omega{\ts  \frac{a^3(x')}{12}} \big| H'\big|^2\, dx'. \label{eq:split_curr}
\end{align}

We conclude that
$$
\liminf_{n\to \infty} I_{\eps_n,\kappa_n}(u_n,A_n) 
	\geq \|a(x')J_3\|_{\Meas(\omega)}
	 + \frac{1}{2} \int_\omega
	  a(x')\left( \left|j'_* - B'_*\right|^2 
	  + {\ts  \frac{a^2(x')}{12}}
	   \big| H'\big|^2 \right) \, dx' \\
$$
with 
\begin{equation}\label{Bstar}
B'_* := {\Aexp}+{\ts \left(\frac{f+g}{2}\right)} (H_2,-H_1).
\end{equation}
Since $\nabla'\times B'_*= F_*$ with $F_*$ as given in \eqref{Fdef},
this concludes the proof.
\end{proof}

\begin{remark}\label{remark}
The fact that the same $\Gamma$-limit is obtained regardless of how the two parameters $\eps_n\to 0$ and $\kappa_n\to\infty$ is somewhat remarkable.  In fact, this is particular to the case where $h^{ex}=O(\log\kappa)$.  To see this, take a rectangular solid domain $\Omega_\eps=(0,a)\times (0,b)\times (0,\eps)$ with horizontal applied field $\mathbf{h}^{ex}=({H_1(\kappa)\over\eps},0,0)$.  The two-dimensional minimizer of the Ginzburg-Landau energy $(w_\kappa(x_2,x_3), A_\kappa(x_2,x_3))$ may be used as a test function, with energy (see Theorem~8.1 of \cite{SS})
$$   \mathbf{I}_{\eps,\kappa}(w_\kappa,A_\kappa) \simeq
    |\Omega_\eps| {\mathbf{h}^{ex}\over 2}\log {\kappa\over\sqrt{|\mathbf{h}^{ex}|}} \simeq  H_1(\kappa)\log {\kappa^2\eps\over H_1(\kappa)}.
$$
A ``thin film'' solution, that is, one with $u=u(x_1,x_2)$ and $A=\Aex$ has energy given by \eqref{tfilm}, with effective field $F=0$, that is
$$   \eps\,G_\kappa(u)=\eps\,\int_\omega a(x')\left\{ \frac12 |\nabla u|^2
            + {\kappa^2\over 4} (|u|^2-1)^2 + 
      {(a(x'))^2\over 24}|{H_1(\kappa)}|^2\,|u|^2\right\} dx'.
$$
For $\kappa$ large, the minimizer is essentially $u=1$, and so a thin film configuration has energy of the order of $\eps H_1(\kappa)^2$.  Thus, if
$$  H_1(\kappa)\log {\kappa^2\eps\over H_1(\kappa)} \ll
                     \eps H_1(\kappa)^2, $$
we expect that the minimizers do not have thin film form. To take a concrete example, if the magnetic field $H_1(\kappa)=\kappa$, 
then the horizontal vortex lattice is preferable for $\eps\gg{1\over\kappa}$ with $\kappa$ large.  
\end{remark}

\section{The recovery sequence}

This section is devoted to proving Theorem~\ref{Glimsup}, namely the existence of a recovery sequence and the $\Gamma-\limsup$ inequality.
In both this section and the following one we require the following Hodge decomposition with respect to the weighted inner product,
$$  \langle v,w\rangle = \int_\omega a(x) \, v\cdot w\, dx'  $$
on $L^2(\omega;\R^2)$.  We define the following subspaces:
\begin{align}
\mathcal U &= \left\{ -{1\over a}\nabla^\perp\psi, \ \psi\in H^1_0(\omega;\R)\right\}, \notag\\
\mathcal V &= \left\{ \nabla \zeta, \ \zeta \in H^1(\omega;\R) \right\}, \label{hodgespaces}\\
\mathcal W &= \left\{ W\in C^1(\omega;\R^2), \, \ \nabla^\perp\cdot W=0, \ 
     \nabla\cdot(aW)=0, \ W\cdot\nu=0 \ \text{on} \ \partial\omega\right\}.\notag
\end{align}

\begin{lemma}\label{hodgelemma}
Any $Z\in L^2(\omega;\R^2)$ admits a unique orthogonal decomposition
$Z= U + V + W$ with $U\in \mathcal U$, $V\in\mathcal V$, $W\in\mathcal W$, with respect to the inner product $\langle\cdot,\cdot\rangle$.  The space $\mathcal W$ is finite dimensional:  in case $\omega$ is simply connected, $\mathcal W=\{0\}$, and in case $\omega=\omega_0\setminus\cup_{i=1}^m \omega_i$ has $m$ holes, $\dim(\mathcal W)=m$.
\end{lemma}

\begin{proof}[]  First, assume  $Z\in C^\infty(\omega;\R^2)$.  
We define $\psi$ and $\zeta$ as the solutions to the boundary-value problems,
$$
\left\{ \begin{aligned}  
    -\nabla\cdot \left({1\over a(x)}\nabla\psi\right) 
       &= \curl Z
    \ \ \text{in $\omega$,} \\
    \psi &= 0 
    \ \ \text{on $\partial\omega$,}
\end{aligned}\right.
\qquad
\left\{ \begin{aligned}  
    \nabla\cdot \left(a(x)\nabla\zeta\right) &= 
      \div [aZ]
    \ \ \text{in $\omega$,} \\
    {\partial\zeta\over\partial\nu} &= Z\cdot\nu 
    \ \ \text{on $\partial\omega$,}
\end{aligned}\right. 
$$
Then, it is easy to verify that 
$W:= Z + \frac1{a}\nabla^\perp\psi - \nabla\zeta$ satisfies $\curl W=0=\div [aW]$ in $\omega$, and $W\cdot\nu =0$ on $\partial\omega$.   Moreover, by integration by parts we see that $W\perp {1\over a}\nabla^\perp\psi\perp \nabla\zeta$ in the inner product $\langle\cdot,\cdot\rangle$.

To identify the space $\mathcal{W}$, we apply Lemma~1.1 of \cite{BBH} and note that any $W\in\mathcal{W}$ may be written as $W={1\over a}\nabla^\perp\xi$ with $\xi$ constant on each component of $\partial\omega$, and $\nabla\cdot{1\over a}\nabla \xi=0$ in $\omega$.  If $\omega$ is simply connected, the maximum principle ensures that $\xi$ is constant in $\omega$, and $\mathcal{W}=\{0\}$ is trivial.  In case $\omega=\omega_0\setminus\cup_{i=1}^m \omega_i$ is multiply connected, we follow the treatment of \cite{AB1}.  For each fixed $i=1,\dots,m$ we define functions $\xi_i\in H^1_0(\omega)$ which solve
\begin{equation}\label{xii}
\left.\begin{gathered}
 \nabla\cdot{1\over a}\nabla \xi_i = 0, \quad\text{in $\omega$,}\\
 \xi_i|_{\partial\omega_j} = c_{ij}, \quad j=1,\dots,m \\
 \xi_i|_{\partial\omega_0}=0, \\
 {1\over 2\pi}\oint_{\partial\omega_j} {1\over a}{\partial\xi_i\over\partial \nu} dx = \delta_{i,j}\quad j=1,\dots,m,
\end{gathered}\right\}
\end{equation}
where $c_{ij}$ are constants (determined by the solutions,) and $\delta_{i,j}$ is Kronecker's delta.
The existence of such $\xi_i$ may be obtained by minimizing
$$  F_i(\xi) = \frac12\int_\omega {1\over a} |\nabla\xi|^2\, dx 
    + 2\pi \xi|_{\omega_i}  $$
over the class of $\xi\in H^1_0(\omega_0)$ with $\xi|_{\omega_j}$ constant.  (See section I.1 of \cite{BBH}.)
It is easy to show that
$$ \xi= \sum_{i=1}^m \Phi_i\xi_i(x), 
\qquad \Phi_i:= 
\left({1\over 2\pi}\oint_{\partial\omega_i} {1\over a}{\partial\xi\over\partial \nu} dx\right) .
$$
Thus, $W={1\over a}\nabla^\perp\xi\in \mathcal{W}$ is parametrized by the $m$ constants 
$\Phi_i$, $i=1,\dots,m$, and  $\mathcal{W}$ is $m$-dimensional.  By elliptic regularity we also have $\mathcal{W}\subset C^1(\omega;\R^2)$.

For general $Z\in L^2(\omega;\R^2)$, the general result is obtained by density. 
\end{proof}

We are now ready to complete the proof of the Gamma convergence result.

\begin{proof}[ of Theorem~\ref{Glimsup}]
Let $j\in\mathcal{Z}$ be given, as well as a sequence $\kappa_n\to \infty$.
We choose vector potentials $A_n=A_{ex}$, and construct a sequence of order parameters $u_n$ of the form $u_n(x)=v_n(x')$ to satisfy the demands of the theorem.  As noted in \cite{AlamaBronsardGS}, for configurations of this form, the three-dimensional energy
reduces,
$$   \tilde I_{\eps_n,\kappa_n}(u_n,A_n) =  G_\kappa(v_n; F_*),  $$
where $G_\kappa$ is defined in \eqref{tfilm} and $B'_*$ is as in \eqref{Bstar}.  Since all which follows will be two-dimensional, we drop the primes in our notation, and write $B_*\in\R^2$, $\nabla=(\partial_{x_1},\partial_{x_2})$, and $v(x), a(x)$ for $x=(x_1,x_2)\in\R^2$.

We next apply the Hodge decomposition above to our given $j\in\mathcal{Z}$, and write
$$   j = U+ V + W = -{1\over a}\nabla^\perp\psi + \nabla\zeta + W, $$
with $\psi\in H^1_0(\omega)$, $\zeta\in H^1(\omega)$ and $W\in\mathcal{W}$, a mutually orthogonal splitting in the inner product 
$\langle\cdot,\cdot\rangle$.  Since $V,W$ are irrotational, they do not contribute to the weak Jacobian $J=\frac12 \nabla\times j$, and carry no vorticity.  As in \cite{JS}, we may associate to $V,W$ an $S^1$-valued map $w^\kappa$.  The singular part of the Jacobian is contained in $U$; for this part we construct a family $u^\kappa$ with point vortices via an appropriate Green's function.  We adapt the arguments of \cite{SSfbp} to deal with the inhomogeneity of the functional $\II$.  Putting these two parts together, the desired recovery sequence will have the form $v_n=w^{\kappa_n}u^{\kappa_n}$.

\paragraph{Step 1: } The components  $V+W\in\mathcal{V}\oplus\mathcal{W}$.

From the proof of Lemma~\ref{hodgelemma}, we may write
$V=\nabla\zeta$, $\zeta\in H^1(\omega)$ and $W={1\over a}\nabla^\perp\xi$ with $\xi(x)=\sum_{i=1}^m \Phi_i\xi_i(x)$, for $\xi_i$ as in \eqref{xii} with $\Phi_i$ real constants.  Let $M_{i,n}=[\Phi_i\ln\kappa_n]$, $i=1,\dots,m$, where brackets denote the integer part, and set
$$  \Xi_n:= \sum_{i=1}^m M_{i,n} \xi_i, \qquad
         W_n= -{1\over a}\nabla^\perp\Xi_n.  $$
We note that 
\begin{equation}\label{West}
\| W_n - W\ln\kappa_n\|_{C^1}\le C,
\end{equation}
for constant $C$ depending on $W$ (but independent of $n$.)

Since 
$$\curl W_n= \sum_{i=1}^m M_{i,n} 
\nabla^\perp\cdot{1\over a}\nabla^\perp\xi_i = 0, \quad\text{and}\quad
\oint_{\partial\omega_j}W_n\cdot\tau\, ds
  = \sum_{i=1}^m M_{i,n} \oint_{\partial\omega_j} 
   {1\over a}{\partial\xi_i\over\partial\nu} ds = 2\pi M_{j,n},
$$
an integer multiple of $2\pi$ for each $j=1,\dots,m$, it follows that $W_n$ is locally a gradient, $W_n=\nabla \eta_n$ for $\eta_n$ possibly multiple valued, but for which $e^{i\eta_n}$ is smooth and single-valued in $\omega$.  We may then define the complex order parameter 
$$
w_n=\exp{i(\eta_n + \zeta\ln\kappa_n)}.
$$
By construction,
\begin{equation}\label{w-conv}
   {j(w_n)\over\log\kappa_n} = 
{(iw_n,\nabla w_n)\over\log\kappa_n} \to V + W
\end{equation}
in $C^1(\bar\omega)$.
Since $|w_n|=1$, we may easily calculate the contribution to the energy using the orthogonality:
\begin{align} \nnn \frac12\int_\omega a(x) 
    |\nabla w_n|^2 \, dx
   &= \frac12 \int_\omega a(x) |\nabla\eta_n + \nabla \zeta\ln\kappa_n|^2\, dx \\
        \nnn
    &= \frac12 \int_\omega a(x) |W_n|^2 +
     \frac{(\ln\kappa_n)^2}2 \int_\omega a(x) |\nabla\zeta|^2\, dx\\
     &\le \frac{(\ln\kappa)^2}2 \int_\omega   a(x)\left\{
       |W|^2 + |V|^2\right\} dx + O(1),
      \label{w-est}
\end{align}
using \eqref{West} in the last line.  This completes Step 1.

\medskip

The treatment of the component $U=-{1\over a}\nabla^\perp\psi\in\mathcal{U}$ will require several steps.  First, we restrict to $\psi\in C_0^\infty(\omega)$; the result for general $\psi\in H^1_0(\omega)$ will follow from a diagonal argument.  Denote by $K\Subset \omega$ the support of $\psi$.

\paragraph{Step 2: }  Approximating the measure $\mu:=\curl U=-\nabla\times{1\over a}\nabla^\perp\psi$ by Dirac masses (representing vortices.)

Let $N_n\in\NN$ be any sequence of whole numbers with
$$  {N_n\over \log\kappa_n}\longrightarrow 1.  $$
Applying Lemma~7.5 of \cite{JS}, there exist families of points 
$\{p_i^n\}_{i=1,\dots,N_n}$ in the set $K=\supp \psi$ and associated integers $\sigma_i^n\in\{-1,1\}$ with the following properties:
\begin{gather}\label{JS1}
|p_i^n - p_j^n| \ge c_0 N_n^{-1/2}\quad \text{for $i\neq j$, for constant $c_0=c_0(\psi)$}; \\
\label{JS4}
\lim_{\alpha\to 0}R(\alpha)=0 \quad
\text{where}\quad R(\alpha)= \limsup_{n\to\infty} 
    \sum_{i\neq j:\atop |p_i^n-p_j^n|\le\alpha}
       {\left|\log |p_i^n - p_j^n|\right| \over N_n^2}, \\
\label{JS2}
\mu_n:= {2\pi\over N_n}\sum_{i=1}^{N_n}
         \sigma^n_i\, \delta_{p_i^n} \wto \mu, \\
\label{JS3}         
|\mu_n|= {2\pi\over N_n}\sum_{i=1}^{N_n}
          \delta_{p_i^n} \wto |\mu|, 
\end{gather}
where the convergence in \eqref{JS2},\eqref{JS3} is weakly in the sense of measures, and strongly in $[C_0^{0,\gamma}]'$ for all $0<\gamma\le 1$.  
By $|\mu|$ we mean the total variation of the measure $\mu=\curl U$.

As in \cite{SSfbp} we modify the measures $\mu_n$ by regularizing the Dirac mass.  Let  $\mu_i^n := \kappa_n\mathcal{H}^1\restr{\partial B(p_i^n,1/\kappa_n)}$, the element of arclength on $S_i^n:=\partial B(p_i^n,1/\kappa_n)$, normalized with mass $2\pi$.  We define the measures
$$
\nu_n  ={1\over N_n} \sum_{i=1}^{N_n} \sigma_i^n\, \mu_i^n,
$$
with
$p_i^n\in K$, $\sigma_i^\kappa\in\{0,1\}$ as above.  Since each 
$\mu_i^n \longrightarrow \delta_{p_i^n}$ strongly in $[C_0^{0,\gamma}(\omega)]'$ for all $0<\gamma\le 1$, and weakly in $\Meas(\omega)$, we may conclude that \eqref{JS2},\eqref{JS3} hold as well for $\nu_n$,
\begin{equation}
\label{JSnu}
\nu_n \longrightarrow \mu, \qquad
|\nu_n|\longrightarrow |\mu|, \quad\text{strongly in $[C_0^{0,\gamma}(\omega)]'$ and weakly in $\Meas(\omega)$}.
\end{equation}
By Fubini's theorem we also note that the product measures also converge,
\begin{equation}\label{Fubini}
\nu_n\otimes \nu_n \longrightarrow \mu\otimes\mu,
\end{equation}
strongly in $[C_0^{0,\gamma}(\omega\times\omega)]'$ and weakly in $\Meas(\omega\times\omega)$.

\paragraph{Step 3:} Recovering $\mathcal{U}$ from $\mu=\curl U$.

We introduce the Dirichlet Green's function,
$G_a(x,y)$ in $\omega$, which solves
$$ \left\{
\begin{gathered}
-\nabla_x\cdot{1\over a(x)}\nabla_x G_a(x,y)= \delta_y(x) , \quad\text{in $\omega$,}\\
G_a(\cdot,y) =0, \quad\text{on $\partial\omega$,}
\end{gathered} \right.
$$
for each fixed $y\in\omega$.  By standard elliptic theory (recall $a>0$ is smooth in $\overline\omega$) we may conclude that $G_a(x,y)$ is smooth in $\overline\omega\times\overline\omega\setminus\{y=x\}$, and
\begin{equation}\label{Green1}
 G_a(x,y) = -\frac{a(x)}{2\pi} \ln |x-y| + \gamma(x,y),  
 \end{equation}
where the regular part $\gamma$ has the property that for every compact set $K\Subset\omega$, there exists $C(K)<\infty$ with
$$  \sup_{y\in K\atop x\in\overline\omega} |\gamma(x,y)|\le C(K).  $$

Given $U\in \mathcal{U}$, we then obtain the potential function $\psi\in H^1_0(\omega)$ from $\curl U=\mu$ by solving
$$
\begin{cases}
-\nabla\cdot{1\over a(x)}\nabla \psi = \mu & \text{ in } \omega, \\
\psi = 0 & \text{ on } \partial \omega,\end{cases}
$$
and we recover $U=-{1\over a}\nabla^\perp\psi$.
Using the Green's function representation, we have
$$
\psi(x) = \int_\omega G_a(x,y) \, d\mu(y).
$$
Since $\mu\in H^{-1}(\omega)$, we may calculate the weighted norm of $U$ in terms of the measure $\mu$ as follows:
\begin{align}\nonumber
\int_\omega a(x)\, |U|^2\, dx &=
\int_\omega \frac{1}{a} \left|\grad \psi\right|^2 \, dx  \\
\nonumber
	& = -\int_\omega \psi \cdot \grad^\perp \left(\frac1a \grad^\perp \psi\right) \, dx  \\
	\nonumber
	& = \int_\omega \psi(x)\, d\mu(x) \\
	& = \int_\omega \int_\omega G_a(x,y) \, d\mu(y) \, d\mu(x).
	\label{s2}
\end{align}

\paragraph{Step 4:}  There exists a sequence $\psi_n\in H^1_0(\omega)$ for which $-{1\over a}\nabla^\perp\psi_n\longrightarrow U$ strongly in $L^p(\omega)$ for all $p<2$, and
\begin{equation}\label{s3}
  \limsup_{n\to\infty}
\int_\omega {1\over a}|\nabla \psi_n|^2\, dx 
    \le 
\int_\omega a(x)\, d|\mu|(x) + \int_\omega a(x) |U|^2\, dx.
\end{equation}

\medskip

For each $n$, we define $\psi_n(x) = \int_\omega G_a(x,y) \, d\nu_n(y)$, and so $\psi_n$ solves
$$
\begin{cases}
-\nabla\cdot{1\over a(x)}\nabla \psi_n = \nu_n & \text{ in } \omega, \\
\psi_n = 0 & \text{ on } \partial \omega.\end{cases}
$$
By \eqref{JSnu} and elliptic regularity, we have $\psi_n\to\psi$ in $W^{1,p}(\omega)$ for all $p<2$, and thus $-{1\over a}\nabla^\perp\psi_n\to U$ in $L^p(\omega)$ for all $p<2$ as claimed.

To estimate the energy we use the Green's representation.
Since $\nu_n\in H^{-1}(\omega)$ for fixed $n$, by \eqref{s2} we conclude that
$$
\int_\omega \frac{1}{a} \left|\grad \psi_n\right|^2 \, dx 
	= \int_\omega \int_\omega G_a(x,y) \, d\nu_n(y) \, d\nu_n(x).
$$
For any $0<\alpha<1$, let $\Delta_\alpha = \{ (x,y) \in \omega\times\omega: |x-y|\leq \alpha\}$.  Fix $\chi_\alpha\in C^\infty(\bar\omega\times\bar\omega)$ with $0\le\chi_\alpha\le 1$, and
$$
\chi_\alpha(x,y)=
\begin{cases}
1, &\text{if $x\in \Delta_\alpha$},\\
0, &\text{if $x\notin \Delta_{2\alpha}$}.
\end{cases}
$$
For any $\alpha\in (0,1)$, $G_a(x,y)(1-\chi_\alpha(x,y))$ is smooth, and hence by the strong $[C_0^{0,\gamma}]'$ convergence $\nu_n\to\mu$ we have:
\begin{equation}\label{s3a}
\lim_{n\to\infty} \int_\omega\int_\omega G_a(x,y)(1-\chi_\alpha(x,y))
 d\nu_n(y)\, d\nu_n(x) =
   \int_\omega\int_\omega G_a(x,y)(1-\chi_\alpha(x,y))
 d\mu(y)\, d\mu(x) .
\end{equation}

For the complementary integral, we use \eqref{Green1} to observe that
\begin{align}\nonumber
\int_\omega\int_\omega G_a(x,y)\chi_\alpha(x,y)
 d\nu_n(y)\, d\nu_n(x) &=
  \int_K\int_{\Delta_{2\alpha}} 
  \left[ {a(x)\over 2\pi} \log{1\over |x-y|} + \gamma(x,y)\right] \chi_\alpha\,d\nu_n(y)\, d\nu_n(x) \\
  \nonumber
  &\le \int_K\int_{\Delta_{2\alpha}} {a(x)\over 2\pi} \log{1\over |x-y|}
   \,d\nu_n(y)\, d\nu_n(x) + C\alpha \\
   \label{s3b}
 &={1\over N^2_n} \sum_{i,j=1}^{N_n} 
 \iint_{\Delta_{2\alpha}} {a(x)\over 2\pi} \log{1\over |x-y|}
   \,d\mu_i^n(y)\, d\mu_i^n(x) + C\alpha.
\end{align}

To evaluate the remaining integral, we consider the contribution due to distinct points $p_i^n\neq p_j^n$ in $\Delta_{2\alpha}$ separately. We adapt an argument in Proposition~7.4 of \cite{SS}.  
Define the index set 
$$\mathcal{J}_n=\{(i,j): \ |p_i^n-p_j^n|\le 2\alpha\}.  $$
Let $R_n=\frac14 c_0 N_n^{-1/2}$, where $c_0=c_0(\psi)$ is the constant in \eqref{JS1}.  We also define balls $\tilde B_i^n=B(p_i^n, R_n)$, $i=1,\dots,N_n$.  By the choice of $R_n$, they are disjoint, as is the union
$$  \bigcup_{(i,j)\in\mathcal{J}_n} \left(\tilde B_i\times\tilde B_j \right)
   \subset \Delta_{3\alpha}.
$$
We also observe that for any $R\le R_n$ and $(i,j)\in\mathcal{J}_n$, since $R\le \frac14|p_i-p_j|$, we have
\begin{equation}\label{s3f}
\frac12 \le {|x-y|\over |p^n_i-p^n_j|} \le \frac32 \quad
  \text{for all $x\in B(p_i^n,R)$, $y\in B(p_j^n,R)$}.
\end{equation} 

For $(i,j)\in\mathcal{J}_n$ we then have (recalling that $S_n^i=\partial B(p_i^n,{1\over\kappa_n})=\supp\mu_i^n$,)
\begin{align*}
\iint_{\tilde B_i^n\times B_j^n} \log{3\over |x-y|} dx\, dy
&\ge  
  \iint_{\tilde B_i^n\times B_j^n} \log{2\over |p^n_i-p^n_j|} dx\, dy
  \\
 & = \pi^2 R_n^4  \log{2\over |p^n_i-p^n_j|}  \\
 & = {R_n^4\over 4} \iint_{\tilde S_i^n\times S_j^n}
    \log{2\over |p^n_i-p^n_j|} d\mu_i^n(x) \, d\mu_j^n(y) \\
 &\ge {R_n^4\over 4} \iint_{\tilde S_i^n\times S_j^n}
    \log{1\over |x-y|} d\mu_i^n(x) \, d\mu_j^n(y),
\end{align*}
using \eqref{s3f} in the first and last lines.  Summing over all pairs $(i,j)\in\mathcal{J}_n$, and using the disjointness of the union of the $\tilde B_i^n\times \tilde B_j^n$, we obtain:
\begin{align} \nonumber
{1\over N_n^2} \sum_{(i,j)\in\mathcal{J}_n} \iint_{S_i^n\times S_j^n}
     {a(x)\over 2\pi} \log{1\over |x-y|} d\mu_i^n(x) \, d\mu_j^n(y)
  &\le {C\over R_n^4\, N_n^2} \sum_{(i,j)\in\mathcal{J}_n}
     \iint_{\tilde B_i^n\times B_j^n} \log{3\over |x-y|} dx\, dy
     \\
  & \le C \iint_{\Delta_{3\alpha}} \log{3\over |x-y|} dx\, dy =: \mathcal{R}(\alpha).
  \label{s3d}
\end{align}
As $|\log|x-y||$ is integrable, the remainder $\mathcal{R}(\alpha)\to 0$ as $\alpha\to 0$, and so this term will not contribute to the limiting energy.

Finally, we consider the contribution from the self-energy of the vortices $p_i^n$.
We parametrize the integrals over 
$S_i^n=\partial B(p_i^n,{1\over\kappa_n})$ using complex notation, that is we
write $x,y\in \partial B(p_i^n,{1\over\kappa_n})$ as 
$x=p_i^n+{1\over\kappa_n}e^{i\theta}$, $y=p_i^n+{1\over\kappa_n}e^{i\tau}$, $0\le\theta,\tau<2\pi$.  Then we have:
\begin{align*}
\nonumber
{1\over N^2_n}
\iint_{\omega}
    {a(x)\over 2\pi} \log{1\over |x-y|}\,d\mu_i^n(y)\, d\mu_i^n(x)
&= {1\over N^2_n}\int_0^{2\pi}\int_0^{2\pi}
    {a\left(p_i^n+ {e^{i\theta}\over\kappa_n}\right)\over 2\pi}
     \left[ \log\kappa_n + \log\left| e^{i(\theta-\tau)}-1\right|\right]
      d\theta\,d\tau\\
 \nonumber
&=   {1\over N_n} \int_0^{2\pi}
        a\left(p_i^n+ {e^{i\theta}\over\kappa_n}\right)\, d\theta + O(N_n^{-2}) \\
&=    {1\over N_n} \int_\omega a(x)\, d|\mu_i^n|(x) + O(N_n^{-2}).
\end{align*}
Summing over all $i=1,\dots,N_n$, we arrive at
\begin{align}\nonumber
  {1\over N^2_n}\sum_{i=1}^{N_n} 
\iint_{\omega}
    {a(x)\over 2\pi} \log{1\over |x-y|}\,d\mu_i^n(y)\, d\mu_i^n(x)
    &= {1\over N_n} \int_\omega a(x)\, d|\nu_n|(x) + O(N_n^{-1})
    \\
    \label{s3c}
   &= \int_\omega a(x)\, d|\mu|(x) + O(N_n^{-1}).
\end{align}
Passing to the limit $\kappa_n\to\infty$, we thus obtain from \eqref{s3a},\eqref{s3b},\eqref{s3d}, and \eqref{s3c}, that
\begin{multline*}  \limsup_{n\to\infty}
   \int_\omega\int_\omega G_a(x,y) d\nu_n(y)\, d\nu_n(x)
    \\
        \le 
    \int_\omega a(x)\, d|\mu|(x) + \int_\omega\int_\omega G_a(x,y)(1-\chi_\alpha(x,y))
 d\mu(y)\, d\mu(x) + C\alpha + C\mathcal{R}(\alpha).
\end{multline*}
By hypothesis, the measure $\mu$ is bounded and absolutely continuous, and so we may apply dominated convergence to pass to the limit $\alpha\to 0$ and obtain the desired bound \eqref{s3}, as
\begin{align*}  \limsup_{n\to\infty}
\int_\omega {1\over a}|\nabla \psi_n|^2\, dx &=
\limsup_{n\to\infty}
   \int_\omega\int_\omega G_a(x,y) d\nu_n(y)\, d\nu_n(x) \\
    &\le 
\int_\omega a(x)\, d|\mu|(x) + \int_\omega\int_\omega 
   G_a(x,y) d\mu(y)\, d\mu(x)\\
   &= \int_\omega a(x)\, d|\mu|(x) + \int_\omega a(x)\, |U|^2\, dx,
\end{align*}
by \eqref{s2}.

\paragraph{Step 5: } Construction of a sequence $u_n\in H^1_0(\omega;\mathbb{C})$.

Let $U_n=-N_n{1\over a}\nabla^\perp \psi_n$.  Then, 
$\nabla^\perp U_n = N_n \, \grad \cdot \big(\frac1a \grad\psi_n\big) =0$ locally in $\omega\setminus \cup_i^{N_n}B(p_i^n,{1\over\kappa_n})$.  Moreover, if $C$ is a simple closed curve in $\omega\setminus \cup_i^{N_n}B(p_i^n,{1\over\kappa_n})$, we have
$$  \int_C U_n\cdot \tau\, ds  \in 2\pi\,\mathbb{Z},  $$
by the normalization $|d\mu_i^n|=2\pi$.  Thus, we may write $U_n=\nabla\phi_n$ in $\omega\setminus \cup_i^{N_n}B(p_i^n,{1\over\kappa_n})$, with $\phi_n$ which is multiple valued, but for which $\nabla\phi_n$ and $e^{i\phi_n}$ are single-valued in $\omega\setminus \cup_i^{N_n}B(p_i^n,{1\over\kappa_n})$.  

To remove the singularity at each vortex core we define, 
$$
\rho_i^n(x) := \begin{cases}
0 & \text{ if } |x-p_i^n| < \frac1{2\kappa_n}, \\
2\kappa_n |x-p_i^n|-1 & \text{ if } \frac1{2\kappa_n} \leq |x-p_i^n| \leq \frac1{\kappa_n}, \\
1 & \text{ if } |x-p_i^n| >\frac1{\kappa_n},
\end{cases}
$$
and $\rho_n := \ds \prod_{i=1}^{N_n} \rho_i^n$.  A simple computation shows that
$$ \int_\omega a(x)\left\{ \frac12 |\nabla\rho_i^n|^2
  +{\kappa^2_n\over 4} ((\rho_i^n)^2-1)^2)\right\} dx
    \le C_0,
$$ 
with constant $C_0$ independent of $n$.

Now define
$u_n = \rho_n e^{i\phi_n}$, with $\rho_n$, $\phi_n$ as in the preceding paragraphs.  We then have:
\begin{align*}
\int_\omega a(x)\left\{ \frac12|\grad u_n|^2 
  + \frac{\kappa_n^2}{4}\big(|u_n|^2-1\big)^2\right\} dx
	& = \int_\omega a(x)\left\{ \frac12\rho_n^2 |\grad \phi_n|^2  
	+ \frac12|\grad \rho_n|^2 
	+ \frac{\kappa^2}{4}\big(\rho_n^2-1\big)^2\right\} dx \\
	& \leq \frac{N_n^2}2\int_\omega \frac{1}{a(x)} |\grad \psi^\kappa|^2 \,dx + C_0 N_n.
\end{align*}
From \eqref{s3} we then conclude that 
\begin{equation}\label{u-en}
\limsup_{n\to\infty} {1\over (\ln\kappa_n)^2}
   \int_\omega a(x)\left\{ \frac12|\grad u_n|^2 
  + \frac{\kappa_n^2}{4}\big(|u_n|^2-1\big)^2\right\} dx
\le
\frac12 \int_\omega a(x)\, d|\mu|(x) + \frac12 \int_\omega a(x) |U|^2\, dx.
\end{equation}
Since $(\rho_n^2-1)\to 0$ in $L^q$ for all $q<\infty$, we also conclude that
\begin{equation}\label{u-conv}
  {j(u_n)\over N_n} = -{1\over a}\nabla^\perp\psi_n
     + {(1-\rho_n^2)\over a}\nabla^\perp\psi_n \longrightarrow 
        U \quad\text{in $L^p(\omega)$ for all $p<2$.}
\end{equation}

\paragraph{Step 6: } Putting it all together.

This follows as in \cite{JS}; we provide details for completeness.
Write $j\in\mathcal{Z}$ as $j=U+\tilde W$ with $U\in\mathcal{U}$ and $\tilde W= V+W$, $V\in\mathcal{V}$, $W\in\mathcal{W}$.
Let $w_n$ be as defined in Step~1 and $u_n$ as constructed in Step~5, and define $v_n=u_n\,w_n$.
Since $|w_n|=1$, we have
\begin{equation}\label{jconv}
  j(v_n)= j(u_n) + \rho_n^2 j(w_n) \longrightarrow U + \tilde W = j
\end{equation}
in $L^p(\omega)$ for all $p<2$.

To estimate the energy, we again use the fact that $|w_n|=1$ to expand:
$$
{1\over N_n^2}\int_\omega a(x)|\nabla v_n|^2\, dx
  = {1\over N_n^2}\int_\omega a(x) \left\{ 
  |\nabla u_n|^2 + \rho_n^2|\nabla w_n|^2 +  j(u_n)\cdot j(w_n)\right\} dx.
$$
We claim that the last term is negligible. Indeed, from Step~1, ${j(w_n)\over\log\kappa_n}=\nabla\Phi_n$, with $\nabla\Phi_n\to \tilde W$ in $C^1$, and therefore,
\begin{align*}
  {1\over N_n^2}\int_\omega a(x) j(u_n)\cdot j(w_n)\, dx
&= -\int_\omega \nabla^\perp\psi_n\cdot \rho_n^2\nabla\Phi_n\, dx \\
&= -\int_\omega \left[\nabla^\perp\psi_n\cdot \nabla\Phi_n -
  (1-\rho_n^2) \nabla^\perp\psi_n\cdot \nabla\Phi_n\right] dx \\
&= \int_\omega (1-\rho_n^2) \nabla^\perp\psi_n\cdot \nabla\Phi_n \, dx
\longrightarrow 0.
\end{align*}
We calculate,
\begin{align*}
&\limsup_{n\to\infty} {1\over N_n^2} G_\kappa(v_n;F_*) \\
&\qquad =\limsup_{n\to\infty} {1\over N_n^2}
 \int_\omega a(x)\left\{ \frac12 |\nabla u_n|^2 + \frac12 |\nabla w_n|^2 - B_*\cdot j(v_n) + |B_*|^2|v_n|^2 
 + {\kappa_n^2\over 4}(|u_n|^2-1)^2 + {a(x)^2|H'|^2\over 24} \rho_n^2\right\}  \\
&\qquad\le \frac12 \int_\omega a(x) \, d|\mu| 
+ \frac12\int_\omega a(x)\left(|U|^2 + |\tilde W|^2 - B_*\cdot j + |B_*|^2 + {a(x)^2|H'|^2\over 12}\right)\, dx \\
&\qquad = \frac12 \int_\omega a(x) \, d|\mu| 
+ \frac12\int_\omega a(x)\left( |j-B_*|^2 + {a(x)^2|H'|^2\over 12}\right)\, dx \\
&\qquad = I_\infty(j; F_*) + \int_\omega {a(x)^3|H'|^2\over 24} dx,
\end{align*}
with $F_*=\curl B_*$,
where we have used \eqref{u-en}, \eqref{w-est}, and \eqref{jconv}.
The completes the proof of the Gamma convergence result. 
\end{proof}

\section{The Obstacle Problem}

We now examine the Gamma limit and its minimizers.  Since the limiting functional is two-dimensional, we simplify notation somewhat, and denote by $x=(x_1,x_2)\in \R^2$, and use $j$, $B\in\R^2$, and $F=\nabla\times B$ dropping the primes and the asterisks.
We also associate to the 1-form $j$ its representation as a vector field $j=(j_1,j_2)\in\R^2$, and write the Jacobian as a scalar measure,
$J=\frac12\nabla\times j=\frac12[\partial_{x_1}j_2-\partial_{x_2}j_1]$.

Using convex duality, we may identify the minimizers of the limiting energy 
$$ I_\lambda(j; F) = \frac12\left\|a(x)\, \nabla \times j\right\|_{\Meas(\omega)}
   + \frac12 \int_\omega a(x)|j- B|^2 \, dx  $$
 (obtained in Theorem~\ref{Glimsup}) as solutions of a {\em two-obstacle} problem for Poisson's equation.  The first step is to rewrite the minimization problem for $I_\infty$ in terms of a scalar potential function.
 
\begin{lemma}\label{lem:zeta}
The minimizer of $I_\infty$ over $\mathcal{Z}$ is attained at 
$j=  B - {1\over a} \nabla^\perp\zeta$, where $\zeta\in H^1_0(\omega)$ minimizes the functional
$$  E_\infty(\zeta;F)= \begin{cases}
\frac12 \left\|a(x)\left( \nabla\cdot{1\over a}\nabla\zeta 
   +  F\right)\right\|_{\Meas(\omega)} + \frac12 \int_\omega {1\over a(x)}|\nabla\zeta|^2 \, dx,
&\text{if $\nabla\cdot{1\over a}\nabla\zeta 
 +  F\in\Meas(\omega),$}\\
+\infty, &\text{otherwise},
\end{cases}
$$
for $F=\curl B$.
\end{lemma}

We observe that the Jacobian corresponding to the minimizer is given by:
\begin{equation}\label{limjac}
  J= \frac12\left(\nabla\cdot{1\over a}\nabla\zeta 
 +  F\right) .
\end{equation}

\begin{proof}[] 
Both $I_\infty$, $E_\infty$ are convex and lower semicontinuous, so each has a minimizer.  Let $j\in \mathcal{Z}$.  We apply the  Hodge decomposition from Lemma~\ref{hodgelemma} to  $j-B$, to obtain 
$$   j-B = {1\over a}\nabla^\perp\zeta + \nabla\eta + W, $$
orthogonal with respect to the inner product $\langle\cdot,\cdot\rangle$, with $\zeta\in H^1_0(\omega)$, $\eta\in H^1(\omega)$, and $W\in \mathcal{H}$.  By the orthogonality of the decomposition,
$$   I_\infty(j;F) = E_\infty(\zeta)  
+  \frac12 \int_\omega a(x)\left\{ |\nabla \eta|^2 + |W|^2\right\}\, dx. 
$$
In particular, $j$ minimizes $I_\infty$ if and only if the associated $\zeta$ minimizes $E_\infty$, and both $\eta, W\equiv 0$.  
\end{proof}

Note that for minimizers we conclude that $\div (a(x)(j-B))=0$ in $\omega$, and $(j-B)\cdot\nu=0$ on $\partial\omega$.  These two conditions may also be obtained from the Ginzburg--Landau equations by passing to the thin-film limit for minimizers of $I_{\eps,\kappa}$.

\begin{proposition}\label{dual}
Any minimizer $\zeta\in H^1_0(\omega)$ of $E_\infty$ is also a minimizer of 
\begin{equation}\label{obprob2}
\min_{u\in H^1_0(\omega) \atop |u|\le a(x)/2} \mathcal{E}_\infty(u), \quad\text{with}\quad
 \mathcal{E}_\infty(u):=
   \int_\omega \left[{1\over a(x)}   |\nabla u|^2 - F\, u\right].
\end{equation}
\end{proposition}

\begin{proof}[]
We write $E_\infty(\zeta)= \Psi(\zeta)+\Phi(\zeta)$, with 
$$ \Psi(\zeta):= \frac12\int_\omega {1\over a}|\nabla\zeta|^2, \quad \Phi(\zeta):= \frac12 \left\|a(x)\left( \nabla\cdot{1\over a}\nabla\zeta 
   +  F\right)\right\|_{\Meas(\omega)}.
$$   
We then calculate the Legendre transform (conjugate function) of each, with respect to the norm 
$\|u\|= \sqrt{\int_\omega {1\over a} |\nabla u|^2 }$ on $H^1_0(\omega)$.  Clearly, $\Psi^*(u)=\Psi(u)=\frac12 \|u\|^2$.  For $\Phi$, we have:
\begin{align*}
\Phi^*(u) &= \sup_{\zeta\in H^1_0(\omega)}
    \left[  \int_\omega {1\over a} \nabla u\cdot\nabla\zeta
     - \frac12 \left\|a(x)\left( \nabla\cdot{1\over a}\nabla\zeta 
   +  F\right)\right\|_{\Meas(\omega)} \right] \\
   &= \sup_{\zeta\in H^1_0(\omega)}
    \left[  \int_\omega u\left(  -F - \nabla\cdot{1\over a}\nabla\zeta\right)
     - \frac12 \left\|a(x)\left( \nabla\cdot{1\over a}\nabla\zeta 
   +  F\right)\right\|_{\Meas(\omega)} \right] 
      + \int_\omega uF \\
   &= \sup_{v\in H^{-1}\cap\Meas} \left[ \int_\omega uv
     -\frac12 \|av\|_\Meas \right] 
       + \int_\omega uF.
\end{align*}
If $\|u/a\|_\infty\le\frac12$, the bracketed expression above is non-positive for any $v$, and so the supremum is achieved at $v=0$.  On the other hand, for $\|u/a\|_\infty>\frac12$, the bracketed expression is unbounded above.  Thus, we conclude that
$$  \Phi^*(u) = \begin{cases}
   \int_\omega uF, 
        &\text{if $\left\|{u\over a}\right\|_\infty\le\frac12$,} \\
     +\infty, &\text{otherwise.}
\end{cases}
$$
By the Fenchel--Rockefeller Theorem (see \cite{EkeTem} or \cite{Brezis} for instance,) 
$$\min_{H^1_0(\omega)} E_\infty 
   = -\min_{u\in H^1_0(\omega)}  (\Psi^*(u) + \Phi^*(-u))
   = -\min_{u\in H^1_0(\omega) \atop |u|\le a(x)/2} 
   \int_\omega \left[{1\over a(x)}   |\nabla u|^2 - F\, u\right], $$
and the minimizers coincide.  
\end{proof}

The minimization problem \eqref{obprob2} is the two-obstacle problem, 
and minimizers $u\in K$, for 
$$ K:=\left\{u\in H^1_0(\omega): \ |u(x)|\le {a(x)\over 2}\right\}$$
 solve the variational inequality
\begin{equation}\label{varineq}
\int_\omega \left[ \nabla u \cdot \nabla (v-u) - F\,(v-u)\right] \ge 0,
\quad\text{for all $v\in K$.}
\end{equation}
Minimizers solve the Euler-Lagrange equation away from the coincidence sets, where the pointwise constraint is attained, along free boundaries.
By Theorem~3.2 in Chapter 1 of  \cite{Friedman}, there is a unique, regular minimizer to problem \eqref{obprob2}:

\begin{proposition}\label{obstacle}
Let $a\in C^2(\overline{\omega})$ and $F\in C^\alpha(\overline{\omega})$, for some $\alpha>0$.  There exists a unique $\zeta\in H^1_0(\omega)$  which minimizes $E_\infty$. Moreover, $\zeta \in W^{2,q}(\Omega)$ for all $q<\infty$, and solves
\begin{equation}\label{freebdry}
\begin{cases}
\nabla\cdot {1\over a}\nabla \zeta + F=0 & \text{ on } 
\{-{a(x)\over 2}<\zeta <{a(x)\over 2}\} \\
\nabla\cdot {1\over a}\nabla \zeta + F \ge 0 & \text{ on } \{\zeta >-{a(x)\over 2}\} \\
\nabla\cdot {1\over a}\nabla \zeta + F\le 0 & \text{ on } \{\zeta< {a(x)\over 2}\}
\end{cases}
\end{equation}
\end{proposition}

By the regularity (and the Sobolev embedding) $\zeta\in C^1(\overline{\omega})$, and thus $\zeta$ satisfies both a Dirichlet and Neumann condition on the boundary of the coincidence sets, $\{ |\zeta(x)|=a(x)\}$.
We also note that by \eqref{limjac}, the limiting Jacobian for minimizers is thus supported on the two coincidence sets, 
$$  S_-:= \left\{ x\in\omega: \  \zeta = -{a(x)\over 2}\right\},\qquad
  S_+:= \left\{ x\in\omega: \  \zeta = {a(x)\over 2}\right\}.  $$
The set $S_+$ is thus identified with concentrations of positively oriented vortices (at densities on the order of $\log\kappa$) and $S_-$ supports concentrations of antivortices (at densities on the order of $\log\kappa$) for configurations $(u,A)$ with bounded $I_{\eps,\kappa}$.

\bigskip

To illustrate the consequences of Theorems~\ref{thm:big1} and \ref{Glimsup}, and in particular Proposition \ref{dualprob}, we revisit two examples introduced in our study of the $\eps\to 0$ limit in our previous paper \cite{AlamaBronsardGS}.

\subsection{Vortices and antivortices coexisting}

First we consider a film with uniform thickness and parabolic geometry,
$f(x_1,x_2)= x_1^2 + x_2^2-\frac12$, and $g(x_1,x_2)= f(x_1,x_2)+1$,
for $x'=(x_1,x_2)\in B_1=\omega$ the unit disk.  We choose the external field ${\mathbf h}^{ex}= \left(H{\log\kappa\over\eps},0,0\right)$,
and so for this example,
$$  a(x')=1,\qquad{f(x')+g(x')\over 2}= x_1^2 + x_2^2, 
\qquad \text{and}\qquad  
 F= -H\partial_{x_1}\left[{f(x')+g(x')\over 2}\right]=-2Hx_1.  $$
Let $\zeta_H$ be the solution of \eqref{freebdry} with parameter $H$, which we vary according to the external field strength, and define
$$   \xi_H= {1\over 2H} \zeta_H,  $$
which solves
$$
\begin{cases}
\Delta \xi_H - x_1 =0 & \text{ on } 
\{-{1\over 4H}<\xi_H <{1\over 4H}\} \\
\Delta \xi_H-x_1 \ge 0 & \text{ on } \{\xi_H >-{1\over 4H}\} \\
\Delta \xi_H -x_1 \le 0  & \text{ on } \{\xi_H < {1\over 4H}\}
\end{cases}
$$
with $\xi_H\in H^1_0(B_1)$.  Thus we consider a fixed equation and allow the obstacles to vary as the external field strength increases.

First, observe that if $\xi_H(x_1,x_2)$ is a solution, then $\tilde \xi_H(x_1,x_2):=-\xi_H(-x_1,x_2)$ is also a solution. By the uniqueness of solutions, we deduce that $\xi_H = \tilde \xi_H$, and hence the solution is odd in $x_1$:
$$
\xi_H(x_1,x_2) = -\xi_H(-x_1,x_2)
$$
Furthermore, $\xi_H \in H^1_0(\omega^-)$ for $\omega^-:=\omega \cap \{x_1<0\}$ the left half-disk, and thus $\xi_H$ solves the two-obstacle problem in $\omega^-$ with Dirichlet boundary condition,
\begin{gather*}
\int_{\omega^-} {1\over a} \nabla\xi_H\cdot\nabla(v-\xi_H)
\ge -\int_{\omega^-} x_1\, (\xi_H-v), \quad\text{for all $v\in K_H^-$},\\
K_H^-:=\left\{ v\in H^1_0(\omega^-): \ |v(x)|\le {1\over 4H}\right\}
\end{gather*}
We claim that $\xi_H>0$ in $\omega^-$.  Indeed, we use $v=\xi_H^+=\max\{\xi_H,0\}\in K_H^-$ in the above variational inequality.  With $\xi_H^-=\max\{-\xi_H,0\}$ we obtain the contradiction
$$  0\le \int_{\omega^-} |\nabla \xi_H^-|^2 \le \int_{\omega^-} x_1\xi_H^- <0, $$
unless $\xi_H^-=0$ almost everywhere.  Thus, $\xi_H\ge 0$ in $\omega^-$.  Applying the strong maximum principle (see \cite{Friedman}, p. 22,) we have $\xi_H>0$ in $\omega^-$ as claimed.  

By symmetry, we conclude that $\xi_H<0$ in the right half-disk, $\omega^+$, and thus the coincidence sets (if nonempty) $S_+=S_+(H)\subset\omega^-$ and $S_-=S_-(H)\subset\omega^+$.  By restricting to the half-disks, $\xi_H$ solves a one-sided obstacle problem.
This problem may be solved explicitly in case the obstacle is not attained:  for 
$H < 3\sqrt{3}$, the solution for this problem (in polar coordinates) is
$$
\xi_H = \frac18r(1-r^2)\cos \theta,
$$
and the coincidence sets $S_\pm$ are empty.
When $H =3\sqrt{3}$, the solution is also given as above, with coincidence sets consisting of a single point $S_{\pm}(3\sqrt{3})=\big\{\big(\mp\frac{1}{\sqrt{3}},0\big)\big\}$.  This value of $H$ marks the lower critical field, the smallest value for which vortices appear in global minimizers of the energy.  By the comparison principle for variational inequalites (see \cite{Friedman}), the solutions $\xi_H$ are monotone decreasing in the half-disk $\omega^-$, and hence the coincidence sets $S_\pm(H)$ form two
increasing collections, $S_\pm(H)\subseteq S_\pm(\tilde H)$ for $H<\tilde H$, symmetrically placed across the $x_2$-axis.  We illustrate this case in Figure~\ref{fig1}.
Since the limiting vorticity is given by,
$$J=\frac12(-\Delta\zeta_H + H x_1)={H\over 2}(-\Delta\xi_H + x_1) $$
we conclude that $J<0$ on the set $S_-$ and $J>0$ on the set $S_+$, indicating the presence of positively oriented vortices on $S_+$ and antivortices on $S_-$.

\begin{figure}
\begin{center}
\begin{tabular}{ccc}
\includegraphics*[width=160pt]{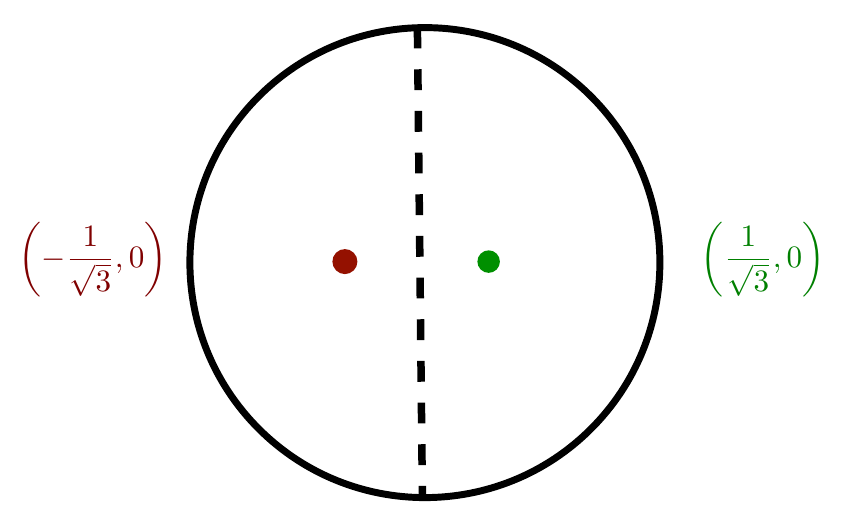}
	& \qquad & \includegraphics*[width=100pt]{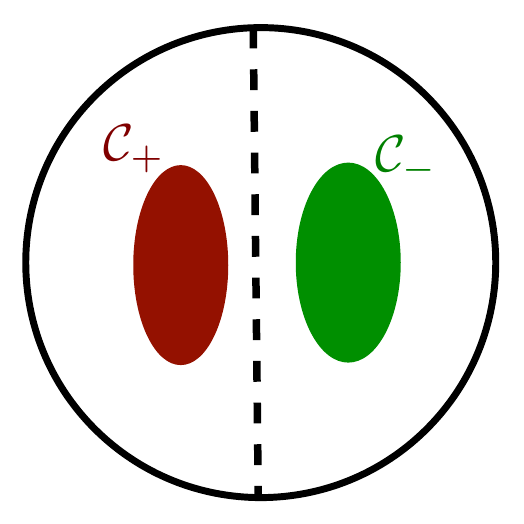} \\
\end{tabular}
\end{center}
\caption{In the first example, vortices first appear for $H=3\sqrt{3}$.  For $H>3\sqrt{3}$, positively charged vortices occupy a region in the left half-disk, and antivortices concentrate in a symmetrically placed domain in the right half-disk.}
\label{fig1}
\end{figure}

\subsection{Vortices accumulating on a ring}

We again choose $\omega=B_1$, the unit disk, and take a film of uniform width $a(x')=1$ defined by the lower surface,
$$
f(x_1,x_2) = -x_2\left(1-4x_1^2-\frac43 x_2^2\right) - \frac12
$$
with $g(x_1,x_2) = f(x_1,x_2) + 1$.  Taking an applied field of the form 
${\mathbf h}^{ex}= H{\log\kappa\over\eps}\vec e_2$ gives an effective field strength (in polar coordinates),
$$  F= (4r^2 -1)H. $$
Making a similar change of variables $\xi_H= {1\over H} \zeta_H$ as in the previous example, we obtain a two-obstacle problem for $\xi_H$ with a fixed equation but obstacles varying in $H$,
$$
\begin{cases}
\Delta \xi_H + 4r^2 -1=0 & \text{ on } 
\{-{1\over 2H}<\xi_H <{1\over 2H}\} \\
-\Delta \xi_H + 4r^2 -1 \ge 0 & \text{ on } \{\xi_H >-{1\over 2H}\} \\
-\Delta \xi_H + 4r^2 -1 \le 0 & \text{ on } \{\xi_H < {1\over 2H}\}
\end{cases}
$$
The beauty of this problem is that it can be solved explicitly, even after the constraints have been attained, and the coincidence sets are determined explicitly as functions of $H$.  Indeed, the general solution of the (unconstrained) equation is $
\xi = \frac14 r^2(1-r^2) + c_1 \log r + c_2
$, and the constants may be calculated to satisfy the obstacle constraints.

We may then assert:
\begin{romanlist}
\item When $H<8$, there is no coincidence set, the solution $\xi_H$ does not contact either obstacle.

\item When $H=8$, the coincidence set $S_+$ is a circle, with radius $r=\frac{1}{\sqrt{2}}$.  (The coincidence set $S_-=\emptyset$.)  The solution is still smooth and is represented in Figure \ref{fig_obstprob}.~(a).

\item For $8 < H <16$, the coincidence set $S_+$ expands to an annulus, with the outside radius increasing while the inner radius remaining fixed in $H$:
$$
S^+ = \left\{ \frac{1}{\sqrt{2}} \leq r \leq R(H)\right\},
$$
with $R(H)$ a strictly increasing function of $H$.  In this interval, the lower obstacle is not attained, so $S_-=\emptyset$.  This is illustrated in figure \ref{fig_obstprob}.~(b).

\item At $H=16$, the solution contacts the lower obstacle, and $S_-=\{(0,0)\}$.  The approximate value of  $R(16)=0.8229681606$. (See figure \ref{fig_obstprob}.~(c).)

\item[(v)] When $H>16$, the coincidence set enlarges on all fronts:
\begin{align*}
S^+ & = \{ \rho^+(H) \leq r \leq R(H)\}, \\
S^- & = \{ r\leq \rho^-(H) \},
\end{align*}
where $0<\rho^-(H)<\frac12<\rho^+(H)<\frac{1}{\sqrt{2}}$.  Each function $\rho^\pm(H),R(H)$ is monotone, with $\rho^\pm(H)\to\frac12$ and $R(H)\to 1$ as $H\to\infty$.  The development of both coincidence sets is illustrated in figure~\ref{fig_obstprob}.~(d) and (e).  As the Jacobian is supported on the coincidence sets, vortices accumulate both in the annular region $\rho^+(H)<r<R(H)$ (with positive orientation), and in the disk $r<\rho^-(H)$ as antivortices (with negative flux.)

\end{romanlist}

\begin{figure}[!htbp]
\begin{center}
\begin{tabular}{ccc}
\includegraphics*[width=150pt]{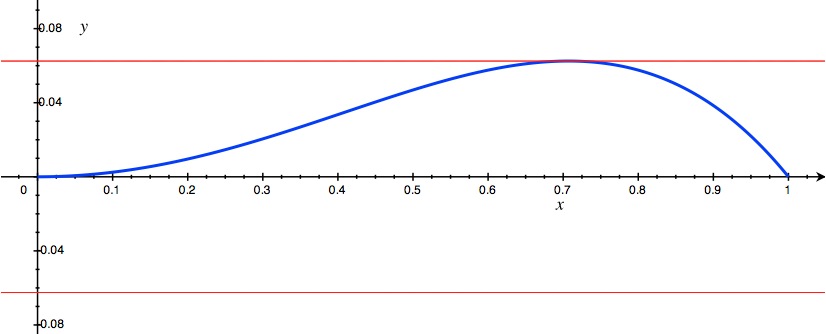}
	& \qquad & \includegraphics*[width=150pt]{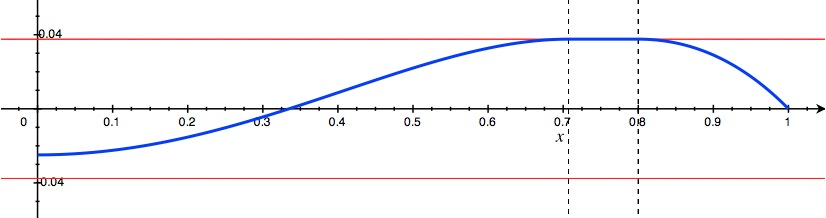} \\
(a) & & (b) \\
\includegraphics*[width=150pt]{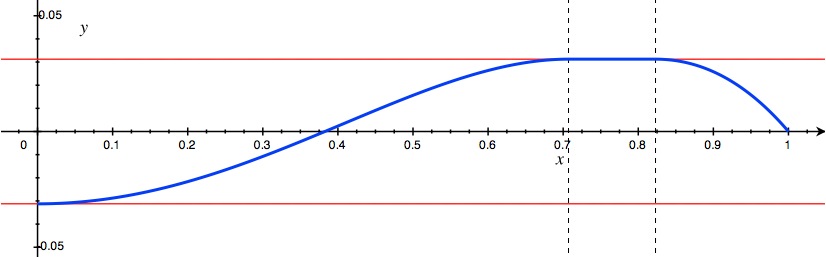}
	& & \includegraphics*[width=150pt]{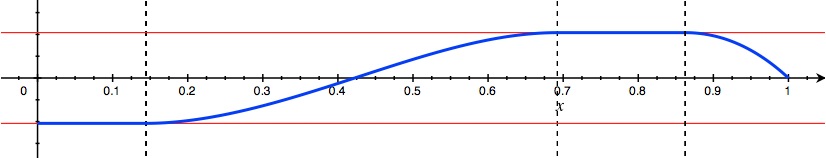} \\
(c) & & (d) \\
\includegraphics*[width=150pt]{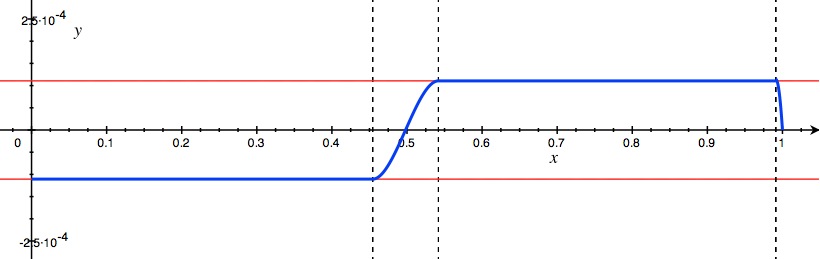}\\
(e) 
\end{tabular}
\end{center}
\caption{The minimizer $\xi_H(r)$ for various values of $H$.
(a) First appearance of vortices for $H=8$; 
	(b) For $8<H<16$, the coincidence set $\{u={1\over 2H}\}$ increases towards the outside only as $H$ increases;
	(c) When $H=16$, the bottom obstacle is reached;
	(d) When $H>16$, both coincidence sets keep increasing;
	(e) As $H \to \infty$, the coincidence sets increase until they touch at $r=\frac12$.}
\label{fig_obstprob}
\end{figure}

\subsection*{Acknowledgements}
The authors were supported by an NSERC (Canada) Discovery Grant.
Part of this work was done while B. Galv\~ao-Sousa was a postdoctoral fellow at McMaster University.


\bibliographystyle{amsalpha}
\addcontentsline{toc}{section}{References}
\bibliography{refs}


\end{document}